\documentclass{amsart}
\usepackage[left=1in, right=1in]{geometry} 
\usepackage[utf8]{inputenc}
\usepackage{amssymb}
\usepackage{amscd,amsfonts,amsmath,amssymb, bbm,dsfont, comment}
\usepackage{amsmath}
\usepackage{amsfonts}
\usepackage{amssymb}\usepackage{enumerate,epsf,fancyhdr,float,graphicx,tabularx}
\usepackage{latexsym,mathrsfs,multirow}
\usepackage{wasysym}
\usepackage{xypic}
\usepackage[all]{xy}\usepackage[OT2,T1]{fontenc}
\usepackage[modulo,mathlines,displaymath, running]{lineno}
\usepackage{amsmath,mathrsfs,enumerate,wasysym}
\usepackage{xypic,hhline}
\usepackage[all]{xy}
\usepackage[OT2,T1]{fontenc}
\usepackage[modulo,mathlines,displaymath]{lineno}
\usepackage{tikz,tikz-cd}
\usepackage{dashrule}
\usetikzlibrary{matrix}
\usepackage[modulo,mathlines,displaymath]{lineno}

\newtheorem{theorem}{Theorem}[section]

\DeclareSymbolFont{cyrletters}{OT2}{wncyr}{m}{n}\DeclareMathSymbol{\Sha}{\mathalpha}{cyrletters}{"58}
\DeclareMathSymbol{\FSha}{\mathalpha}{cyrletters}{"11}

\renewcommand{\phi}{{\varphi}}

\renewcommand{\geq}{\geqslant}
\renewcommand{\leq}{\leqslant}

\newcommand{\HIw}{H^1_{\mathrm{Iw}}}

\newcommand{\Fil}{\mathrm{Fil}}

\newcommand{\ur}{\mathrm{un}}

\newcommand{\links}{\left(\begin{array}{cc}}
\newcommand{\rechts}{\end{array}\right)}
\newcommand{\bai}{\left[\begin{array}{cc}}
\newcommand{\dai}{\end{array}\right]}
\newcommand{\hidari}{\left(\begin{array}{c}}
\newcommand{\migi}{\end{array}\right)}

\newcommand{\C}{\mathbb{C}}

\newcommand{\Q}{\mathbb{Q}}

\newcommand{\Z}{\mathbb{Z}}

\newcommand{\fS}{\mathfrak{S}}

\newcommand{\bJ}{\mathbf{J}}

\newcommand{\bV}{\mathbf{V}}
\newcommand{\bT}{\mathbf{T}}
\newcommand{\bA}{\mathbf{A}}
\newcommand{\bM}{\mathbf{M}}

\DeclareMathOperator{\Gr}{Gr}

\newcommand{\Gal}{\operatorname{Gal}}

\newcommand{\Frob}{\operatorname{Frob}}
\newcommand{\Hom}{\operatorname{Hom}}

\newcommand{\ord}{\operatorname{ord}}
\newcommand{\col}{\operatorname{Col}}

\newcommand{\im}{\operatorname{Im}}
\newcommand{\Sel}{\operatorname{Sel}}

\renewcommand{\O}{\mathcal{O}}

\renewcommand{\contentsname}{Contents\\{\footnotesize\normalfont(A table
of contents should normally not be included)}}

\newtheorem{auxiliary proposition}[theorem]{Auxiliary Proposition}

\newtheorem{conjecture}[theorem]{Conjecture}

\newtheorem{corollary}[theorem]{Corollary}

\newtheorem{definition}[theorem]{Definition}

\newtheorem{lemma}[theorem]{Lemma}
\newtheorem{main conjecture}[theorem]{Main Conjecture}
\newtheorem{main theorem}[theorem]{Main Theorem}
\newtheorem{modesty proposition}[theorem]{Modesty Proposition}

\newtheorem{open problem}[theorem]{Open Problem}

\newtheorem{proposition}[theorem]{Proposition}

\newtheorem{remark}[theorem]{Remark}

\newtheorem{convergence lemma}[theorem]{Convergence Lemma}
\newtheorem{corrected lemma}[theorem]{Corrected Lemma}
\newtheorem{growth lemma}[theorem]{Growth Lemma}
\newtheorem{coefficient lemma}[theorem]{Integrality Lemma}
\newtheorem{interpolation lemma}[theorem]{Interpolation Lemma}
\newtheorem{kernel lemma}[theorem]{Kernel Lemma}
\newtheorem{limit lemma}[theorem]{Limit Lemma}
\newtheorem{tandem lemma}[theorem]{Modesty Lemma}
\newtheorem{zero-finding lemma}[theorem]{Zero-Finding Lemma}

\title[~]{On a parity result for the symmetric square of modular forms with congruent residual representations}
\author[Ray]{Jishnu Ray}
\address[Ray]{Harish-Chandra Research Institute, Chhatnag Road, Jhunsi, Prayagraj 211 019, India.}
\address[Ray]{Homi Bhabha National Institute, Training School Complex, Anushakti Nagar, Mumbai 400 094, India.
}
\email{jishnuray@hri.res.in; jishnuray1992@gmail.com}

\begin{document}

\maketitle
\begin{abstract}
    The parity of Selmer ranks for elliptic curves defined over the rational numbers $\Q$ with good ordinary reduction at an odd prime $p$ has been studied by Shekhar. The proof of Shekhar relies on proving a parity result for  the $\lambda$-invariants of Selmer groups over the cyclotomic $\Z_p$-extension $\Q_\infty$ of $\Q$.
    This has been further generalized for elliptic curves with supersingular reduction at $p$ by Hatley and for modular forms by Hatley--Lei. In this paper, we prove a parity result for the $\lambda$-invariants of Selmer groups over $\Q_\infty$ for the symmetric square representations associated to two modular forms  with congruent residual Galois representations. We treat both the ordinary and the non-ordinary cases. 
\end{abstract}
\section{Introduction}
Suppose $K$ is a number field and fix an odd prime $p$. Let $K_\infty$ be the cyclotomic $\Z_p$-extension of $K$. Let $E_1$ and $E_2$ be two elliptic curves defined over $\Q$ both having good ordinary reduction at the prime $p$. Let $\Sigma$ be a finite set of primes of $K$ containing the primes of bad reduction of $E_1$ and $E_2$, the infinite primes and the primes above $p$. Suppose that $E_1$ and $E_2$ are congruent at $p$, i.e. $E_1[p] \cong E_2[p]$ as representations of the absolute Galois group $G_\Q=\Gal(\overline{\Q}/\Q)$. Furthermore assume that $K=K(\mu_p)$ and  $E_1[p]$ is an irreducible $G_K$-module (and hence $E_2[p]$ will also be an irreducible $G_K$-module). Then, under the additional assumption that the Iwasawa $\mu$-invariant of the Greenberg's $p$-Selmer group $\Sel_{p^\infty}(E_1/K_\infty)$ vanishes (and hence also for $\Sel_{p^\infty}(E_2/K_\infty)$), Shekhar proved that 
\begin{equation}\label{Shekhar}
   \mathrm{corank}_{\O_K}\Sel_{p^\infty}(E_1/K) +|S_{E_1}| \equiv \mathrm{corank}_{\O_K}\Sel_{p^\infty}(E_2/K) +|S_{E_2}| \pmod{2}. 
\end{equation}
Here $S_{E_i}$ is an explicitly determined subset of primes in $\Sigma$ (cf. \cite[Theorem 1.1]{Shekhar16}). Assuming that the Shafarevich-Tate group $\Sha(E_i/K)[p^\infty]$ is finite, this gives the parity of ranks of elliptic curves with equivalent mod-$p$ Galois representations. If $K \subset \Q(\mu_{p^n},m^{1/p^n})$ for $m,n \in \Z_{\geq 1}$ and $K$ is a Galois extension over $\Q$, then the $p$-parity conjecture holds, i.e. 
$$\mathrm{corank}_{\O_K}\Sel_{p^\infty}(E_i/K) \equiv r_{\mathrm{an}}(E_i) \pmod{2}.$$
Here $r_{\mathrm{an}}(E_i)$ is the order of zero of the complex $L$-function $L(E_i/K)$ at $1$ (see \cite[Conjecture 1 and Theorem 4.3]{Shekhar16}). This gives a parity of analytic ranks for congruence mod-$p$ Galois representations.

Let $\lambda(E_i/K_\infty)$ be the Iwasawa $\lambda$-invariant attached to $\Sel_{p^\infty}(E_i/K_\infty)$. The proof of \eqref{Shekhar} relies primarily on proving the following:
\begin{equation}\label{lambda2} \lambda(E_1/K_\infty)+|S_{E_1}| \equiv \lambda(E_2/K_\infty)+|S_{E_2}| \pmod{2}   
\end{equation}
and then using a result proved by Greenberg (cf. \cite[Prop. 3.10]{Greenberg}), which is 
\begin{equation}\label{Green}
  \lambda(E_i/K_\infty) \equiv \mathrm{corank}_{\O_K}\Sel_{p^\infty}(E_i/K) \pmod{2}.  
\end{equation}

Shekhar's result has been generalized to elliptic curves with supersingular reduction at the prime $p$ \cite{Hatley17} and  to modular forms \cite[Theorem 5.7]{HatleyLei19}. Although the result in \cite{HatleyLei19} is written for modular forms non-ordinary at $p$ but essentially the same technique will also work for modular forms ordinary at $p$ with signed Selmer groups replaced by classical $p$-Selmer group using \cite[Theorem 4.3.4 (ii)]{EmertonPollackWeston}. The main goal in this paper is to prove an analogue of \eqref{lambda2} for the symmetric square representations associated to modular forms ordinary at $p$. 
Let $f_i=\sum a_n(f_i)q^n$ (for $i=1,2$) be two normalized new cuspidal eigenforms of the same weight $k$, level $N_i$ (coprime to $p$) and character $\varepsilon_i$. Assume that  their residual representations $\tilde{\rho}_{f_i}$ are isomorphic and are irreducible as $G_\Q$-modules. For any Dirichlet character $\psi$ of conductor coprime to $p$, let $\bV_{f_i,\psi}$ be the symmetric square representation associate to $f_i$ twisted by $\psi$. We can enlarge the coefficient field and choose an extension $L$ over $\Q$ which contains the image of $\psi$, and all the coefficients of $f_1$ and $f_2$.  There is a unique Galois stable lattice which is denoted by $\bT_{f_i,\psi}$ and let $\bA_{f_i,\psi}=\bV_{f_i,\psi}/\bT_{f_i,\psi}.$ Let $\mathfrak{P}$ be a prime of $L$ above $p$. Assume that $f_i$ is ordinary at $\mathfrak{P}$. Assume hypothesis \textbf{(Tor)} (see before Remark \ref{rem:tor}). (One can choose such a nontrivial even Dirichlet character  where it is known that the Selmer group $\Sel_{p^\infty}(\bA_{f_i,\psi}/\Q_\infty)$ is cotorsion (cf. Remark \ref{rem:tor})). 
  Also assume the hypothesis \textbf{(invO)} (see after definition \ref{def:imprimitive}). We prove the following result on the parity of $\lambda$-invariance (cf. Theorem \ref{Mainthm}).
\begin{itemize}
    \item Suppose that the $\mu$-invariant of $\Sel_{p^\infty}(\bA_{f_1,\psi}/\Q_\infty)$ vanishes. Assume that for $i\in \{1,2\}$, $N_i$ is square-free and the conductor $\mathrm{cond}_\ell(\tilde{\rho}_{f_i})=\ell$ for all primes $\ell \mid \mid N_i$.
    
    Then there exists some explicitly computable finite sets $\mathcal{S}_{f_1,\psi}$ and $\mathcal{S}_{f_2,\psi}$  such that $$\lambda(\Sel_{p^\infty}(\bA_{f_1, \psi}/\Q_\infty))+|\mathcal{S}_{f_1,\psi}|\equiv \lambda(\Sel_{p^\infty}(\bA_{f_2, \psi}/\Q_\infty))+|\mathcal{S}_{f_2,\psi}| \pmod{2}.$$
\end{itemize}

 As of now, we don't know how to remove hypothesis that $\mathrm{cond}_\ell(\tilde{\rho}_{f_i})=\ell$ for all primes $\ell \mid \mid N_i$.
  There are also restrictions to deduce the analogue of \eqref{Green} in our case (see remark \ref{rem:obs}).

Suppose now that $f_i$ is non-ordinary at $\mathfrak{P}$ and $a_p(f_i)=0$.  Let $\mathcal{S}$ denote the set of pairs $\{(+,-),(+,\bullet), (-,\bullet)\}.$ For $\fS = (\clubsuit, \spadesuit) \in \fS$, and $\psi$ satisfying the conditions mentioned in \cite[page 3]{BLV}, Büyükboduk--Lei--Venkat defined three signed Selmer groups $\Sel_{\fS}(\bA_{f_i, \psi}/\Q(\mu_{p^\infty}))$ which are conjecturally cotorsion.  Assume that $p\geq 2k+1$ and assume hypothesis \textbf{(invS)} (see before Lemma \ref{lem1}).  Below is the summary of our results.
\begin{itemize}
    \item  The Selmer group $\Sel_{\fS}(\bA_{f_i, \psi}/\Q_\infty)$ contains no proper $\Lambda$-submodule of finite index (Theorem \ref{prop:finite}).
    \item The $\mu$-invariant of $\Sel_{\fS}(\bA_{f_1, \psi}/\Q_\infty)$ vanishes if and only if the $\mu$-invariant of $\Sel_{\fS}(\bA_{f_2, \psi}/\Q_\infty)$ vanishes. When these $\mu$-invariants are trivial, then the imprimitive signed $\lambda$-invariants of $\Sel^{\Sigma_0}_{\fS}(\bA_{f_1, \psi}/\Q_\infty)$ and $\Sel^{\Sigma_0}_{\fS}(\bA_{f_2, \psi}/\Q_\infty)$ coincide
(Theorem \ref{prop: lambda}).
\item  Assume that for $i\in \{1,2\}$, $N_i$ is square-free and the conductor $\mathrm{cond}_\ell(\tilde{\rho}_{f_i})=\ell$ for all primes $\ell \mid \mid N_i$. Also assume the vanishing of the $\mu$-invariants of $\Sel_{\fS}(\bA_{f_1, \psi}/\Q_\infty)$ and $\Sel_{\fS}(\bA_{f_2, \psi}/\Q_\infty)$. Then we have the congruence
    $$\lambda(\Sel_{\fS}(\bA_{f_1, \psi}/\Q_\infty)) +|\mathcal{S}_{f_1,\psi}|\equiv \lambda(\Sel_{\fS}(\bA_{f_2, \psi}/\Q_\infty))+|\mathcal{S}_{f_2,\psi}| \pmod{2}$$
    where $\mathcal{S}_{f_i,\psi}$ is the \textit{same} set of primes as in the ordinary case above (see Theorem \ref{Mainthm2}).
\end{itemize}
The main inputs in this article are the computations done to make the set $\mathcal{S}_{f_i,\psi}$ as explicit as possible (i.e. Lemma \ref{lem:noN}, Lemma \ref{lem: ii} and Lemma \ref{lem: iii}). It might be worth mentioning that the computation of the parity of these error terms are not there in \cite{RaySujathaVatsal} in the ordinary case (in particular see Proposition \ref{prop: mod2}). In the non-ordinary setting, our main input is to use the local condition at $p$ defining these signed Selmer groups over $\Q(\mu_{p^\infty})$ in order to define the signed Selmer groups over $\Q_\infty$ and the finite layers $\Q_{(n)}$ in such a way that the usual control theorem holds (cf. Lemma \ref{lem1}). Note that the local condition at $p$ over $\Q_{(n)}$ is defined here using the local condition at $\Q(\mu_{p^\infty})$ via descending. This is unlike Kobayashi's approach \cite{kobayashi03} where the local condition at the finite layers is given first and then a direct limit was taken to define the signed Selmer group at $\Q_\infty$ resulting in a more difficult control theorem \cite[Theorem 9.3]{kobayashi03}. These two approaches give, in general, two different Selmer groups at the layer $\Q_{(n)}$. Our approach solves the purpose we are interested in.
Further we take certain cyclotomic twists of signed Selmer groups such that some global to local map defining appropriate Selmer condition becomes surjective (cf. Lemma \ref{lem: first} and Lemma \ref{lemma:second}). Taking such twists is a crucial part of the argument, without which the arguments fails. Finally, we analyze the Pontryagin dual of the local condition at $p$ defining the signed Selmer structures and prove a congruence result (Proposition \ref{prop:congprop}). Such an analysis of local condition at $p$ for signed Selmer groups is also need in the proof of Lemma \ref{lemma:second}. All these combined efforts finally help  to prove our main theorems.

\section*{Acknowledgement}
We would like to thank Jeffrey Hatley, Antonio Lei, Aprameyo Pal and Dipendra Prasad for numerous discussions and helpful comments. We gratefully acknowledge support from the Inspire Research Grant, DST, Govt. of India. 
We thank the referee for carefully reading our article and suggesting several improvements.
\section{The symmetric square representation}\label{basics}
Throughout we fix an embedding $\iota_\infty$ of a fixed algebraic closure $\overline{\Q}$ of $\Q$ into $\C$ and also an embedding $\iota_\ell$ of $\overline{\Q}$ into a fixed algebraic closure $\overline{\Q}_{\ell}$ for every prime $\ell$. Let $f=\sum a_n(f)q^n$ be a normalized new cuspidal eigenform of even weight $k\geq 2$, level $N$, nebentypus $\varepsilon$, with coefficients in a number field $L \subset \mathbb{C}$.  Let $p \geq 5$ be a prime such that $p \nmid N$ and let $\mathfrak{P}$
be a prime of the field $L$ above $p$. We assume that $f$ is ordinary at $\mathfrak{P}$ (i.e. $v_{\mathfrak{P}}(a_p(f))=0$). Let $\alpha_p(f)$ be the unique root of the Hecke polynomial at $p$ that lies in $ \mathcal{O}^\times_{L,\mathfrak{P}}.$ Let $\omega_p: G_\Q \rightarrow \Z_p^\times$ be the $p$-adic cyclotomic character and let $\psi$ be a Dirichlet character of conductor $N_\psi$ coprime to $p$. Enlarging $L$ if necessary, we assume that $\psi$ takes values in $L^\times$. 
\begin{theorem}[Eichler, Shimura, Deligne, Mazur-Wiles, Wiles, etc.] \label{thm:modular}
There exists a Galois representation $\rho_f:G_\Q \longrightarrow GL_2(L_{\mathfrak{P}})$ such that 
\begin{enumerate}
    \item For all primes $\ell \nmid Np$, $\rho_f$ is unramified with the characteristic polynomial of the (arithmetic) Frobenius is given by 
    $\mathrm{trace}(\rho_f(\Frob_\ell))=a_\ell(f),$ and $\det(\rho_f(\Frob_\ell))=\varepsilon(\ell)\omega_p(\Frob_\ell)^{k-1}=\varepsilon(\ell)\ell^{k-1}.$ It follows by the Chebotarev Density Theorem that $\det(\rho_f)=\varepsilon\omega_p^{k-1}.$
    \item Let $G_p$ be the decomposition subgroup of $G_\Q$ at $p$. Then,
    $$\rho_f|_{G_p} \sim 
\begin{pmatrix}
\lambda_f^{-1}\varepsilon \omega_p^{k-1} & * \\
~ & \lambda_f
\end{pmatrix}$$
where $\lambda_f$ is the unramified character such that $\lambda_f(\Frob_p)=\alpha_p(f).$
\end{enumerate}
Let $V_f$ denote the representation space of $\rho_f$. Since $G_\Q$ is compact, choose an $\mathcal{O}_{L_\mathfrak{P}}$- lattice $T_f$ which is invariant under $\rho_f$. Let $$\tilde{\rho}_f: G_\Q \longrightarrow GL_2(\frac{\mathcal{O}_{L_\mathfrak{P}}}{\pi_L})$$ be the residual representation attached to $\rho_f$.
\end{theorem}
Throughout we assume that $\tilde{\rho}_f$ is absolutely irreducible so that the choice of the Galois stable lattice $T_f$ is unique.  By part (2) of Theorem \ref{thm:modular}, there exists a $G_p$-stable two step filtration 
$$V_f=\Fil^0V_f \supset \Fil^1V_f \supset 0=\Fil^2V_f$$
such that the action of $G_p$ on the graded pieces $\Gr^iV_f:=\Fil^iV_f/\Fil^{i+1}V_f$
is given as follows. The $G_p$-action of $\Gr^1V_f$ (resp. $\Gr^0V_f$) is given via the character $\lambda_f^{-1}\varepsilon \omega_p^{k-1}$ (resp. $\lambda_f$). Hence the action of $G_p$ on $\Gr^0V_f$ is unramified.

Now consider a basis $v_1$ of $\Fil^1V_f$ and expand it to a basis  $\{v_1,v_2\}$ of $V_f$. Let $\bV_f$ be the symmetric square representation associate to $V_f$. A basis for $\bV_f$ is given by $\{w_{i,j} \mid 1 \leq i \leq j \leq 2\}$ where $w_{i,j}=v_i \otimes v_j + v_j \otimes v_i$. Let $\Fil^2\bV_f=\mathrm{span}\{w_{1,1}\}$ and $\Fil^1\bV_f=\mathrm{span}\{w_{1,1},w_{1,2}\}$. Then $\bV_f$ has a $3$-step $G_p$-stable filtration 
$$\bV_f=\Fil^0\bV_f \supset \Fil^1\bV_f \supset \Fil^2\bV_f \supset 0=\Fil^3\bV_f$$
such that the action of $G_p$ on the $1$-dimensional graded pieces $\Gr^2 \bV_f, \Gr^1 \bV_f$ and $\Gr^0 \bV_f$ is given by the characters $(\lambda_f^{-1}\varepsilon \omega_p^{k-1})^2,\varepsilon \omega_p^{k-1}$ and $\lambda_f^2$ respectively. 

\subsection{The Greenberg Selmer group}\label{sec:notation}
Let $\bA_f:=\bV_f/\bT_f $ where $\bT_f=\mathrm{Sym}^2T_f$. We write $\bV_{f, \psi}$ for the twisted representation $\bV_f \otimes \psi$ with lattice $\bT_{f, \psi}$ and we denote the corresponding $p$-divisible Galois module as $\bA_{f, \psi}$.
Let $\Fil^1 \bA_{f, \psi}$ be the image of $\Fil^1 \bV_{f, \psi}$ under the canonical map $\bV_{f, \psi} \rightarrow \bA_{f, \psi}$.
Let $\Sigma$ be  a finite set of places of $\Q$ that contains $p$, primes that divide the level $N$, primes that divide the conductor $N_\psi$ and  $\infty$. We write $\Q_\Sigma$ for the maximal extension of $\Q$ unramified outside $\Sigma$. Let $\Q_\infty$ be the cyclotomic $\Z_p$-extension of $\Q$ with Galois group $\Gamma$ and Iwasawa algebra $\Lambda$. 
The $p$-primary Greenberg Selmer group 
$\Sel_{p^\infty}(\bA_{f, \psi}/\Q_\infty)$ is defined as 
$$\Sel_{p^\infty}(\bA_{f, \psi}/\Q_\infty):=\ker \Big(H^1(\Q_{\Sigma}/\Q,\bA_{f, \psi}) \xrightarrow{\lambda_{f, \psi}} \prod_{\ell \in \Sigma}\mathcal{H}_{\ell}(\Q_\infty, \bA_{f, \psi})\Big)$$
where $\mathcal{H}_{\ell}(\Q_\infty, \bA_{f, \psi})$ is defined as follows. If $\ell \neq p$, 
$$\mathcal{H}_{\ell}(\Q_\infty, \bA_{f, \psi}):= \prod_{\eta \mid \ell} H^1(\Q_{\infty, \eta}, \bA_{f, \psi})$$
where the product is over the finite set of primes $\eta$ of $\Q_\infty$ lying over $\ell$. Let $\eta_p$ be the unique prime of $\Q_\infty$ lying above $p$ and $I_{\eta_p}$ be the inertia subgroup at $\eta_p$. Then $$\mathcal{H}_p(\Q_\infty, \bA_{f, \psi}):=H^1(\Q_{\infty, \eta_p}, \bA_{f, \psi})/\mathcal{L}_{\eta_p}$$ with 
$$\mathcal{L}_{\eta_p}=\ker\Big(H^1(\Q_{\infty,\eta_p}, \bA_{f, \psi}) \longrightarrow H^1(I_{\eta_p}, \bA_{f, \psi}/\Fil^1 \bA_{f, \psi})\Big).$$
We make the following hypothesis throughout the article. 
\\

\noindent \textbf{(Tor)} $\Sel_{p^\infty}(\bA_{f, \psi}/\Q_\infty)$ is a cotorsion $\Lambda$-module.
\vspace{.2cm}

\begin{remark}\label{rem:tor}
    We note that our $\bA_{f,\psi}$ is the module $A$ as in \cite[Notation 3.2.4]{LZ16} where instead of twisting by $\psi^{-1}$, we are twisting by $\psi$. So suppose  we assume that the five hypotheses in \cite[Section 5.1]{LZ16} are satisfied with $\psi$ there replaced by $\psi^{-1}$. Then \cite[Theorem 5.4.2, second bullet point]{LZ16} shows that \textbf{(Tor)} is satisfied.  We note that our local condition is defined using $\Fil^1 \bA_{f, \psi}$ which is the same as the local condition at $p$ for $H^1_{\mathrm{Gr},1}(\Q_\infty,A)$ as in \cite[Proposition 3.4.4 (i)]{LZ16}. We also note that in \cite{LZ16}, the authors are assuming that $f$ is not of CM type. See also Remark 2.1.2 and remark 4.3.3 of \cite{LZ16} where some details are given on this assumption.
\end{remark}
\begin{definition} \label{def:imprimitive}
    Let $\Sigma_0=\Sigma \setminus \{p, \infty\}$. The $\Sigma_0$-imprimitive Selmer group  is defined as 
    $$\Sel_{p^\infty}^{\Sigma_0}(\bA_{f,\psi}/\Q_\infty)=\ker \Big(H^1(\Q_{\Sigma}/\Q,\bA_{f, \psi}) \xrightarrow{\lambda_{f, \psi}^{\Sigma_0}} \prod_{\ell \in \Sigma\setminus \Sigma_0}\mathcal{H}_{\ell}(\Q_\infty, \bA_{f, \psi})\Big)$$
\end{definition}
Let $(\bT_{f,\psi})^*:=\Hom(\bT_{f,\psi},\mu_{p^\infty}).$
We make the following assumption which is the hypothesis in \cite[Corollary 2.3]{GreenbergVatsal}. 
\begin{enumerate}[~] 
    \item \textbf{(invO)} The group  $H^0(\Q_\infty,(\bT_{f,\psi})^*)$ is finite.
\end{enumerate}

Under these assumptions, the localization map $\lambda_{f, \psi}$ (and hence also ${\lambda_{f, \psi}^{\Sigma_0}}$) is surjective (cf. \cite[Corollary 2.3]{GreenbergVatsal}, see also \cite[Proposition 3.3]{RaySujathaVatsal}). It follows that 
\begin{equation}\label{eq:Sel1}
   \Sel_{p^\infty}^{\Sigma_0}(\bA_{f,\psi}/\Q_\infty)/ \Sel_{p^\infty}(\bA_{f,\psi}/\Q_\infty)\cong \prod_{\ell \in \Sigma_0}\mathcal{H}_{\ell}(\Q_\infty, \bA_{f, \psi}).
\end{equation}
The following result is well-known (cf. \cite[Lemma 3.5]{RaySujathaVatsal}).
\begin{lemma}\label{eq:Lemma1}
    If $ \ell \neq p$, $\mathcal{H}_{\ell}(\Q_\infty, \bA_{f, \psi})$ is cofinitely generated and cotorsion $\Lambda$-module with trivial $\mu$-invariant.
\end{lemma}
Let $\lambda^{\Sigma_0}(\bA_{f,\psi}/\Q_\infty)$  and $\lambda(\bA_{f,\psi}/\Q_\infty)$ be the $\lambda$-invariants for the (dual of) Selmer groups 
$\Sel_{p^\infty}^{\Sigma_0}(\bA_{f,\psi}/\Q_\infty)$ and $\Sel_{p^\infty}(\bA_{f,\psi}/\Q_\infty)$ respectively. It follows that 
\begin{equation}\label{eq: lambda}
    \lambda^{\Sigma_0}(\bA_{f,\psi}/\Q_\infty)=\lambda(\bA_{f,\psi}/\Q_\infty) +\sum_{\ell \in \Sigma_0} \delta_\ell(\bA_{f,\psi}/\Q_\infty)
\end{equation}
where the $\lambda$-invariant of the dual of $\mathcal{H}_{\ell}(\Q_\infty, \bA_{f, \psi})$ is given by $\delta_\ell(\bA_{f,\psi}/\Q_\infty):=\sum_{\eta \mid \ell} \tau_{\eta}(\bA_{f,\psi}/\Q_\infty)$. Here, $\tau_{\eta}(\bA_{f,\psi}/\Q_\infty)$ is the $\lambda$-invariant of  the dual of $H^1(\Q_{\infty, \eta}, \bA_{f, \psi})$.

\subsection{Computing the parity of $\delta_\ell(\bA_{f,\psi}/\Q_\infty)$} For a prime $\ell \in \Sigma_0$, let $\Frob_\ell$ denote the arithmetic Frobenius automorphism in $\Gal(\Q_\ell^{\mathrm{unr}}/\Q_\ell)$. Let $I_\ell$ be the inertia subgroup $ \Gal(\overline{\Q}_l/\Q_\ell^{\mathrm{unr}})$ of $G_{\Q_\ell}$, $(\bV_{f , \psi})_{I_\ell}$ be the maximal quotient of $\bV_{f , \psi}$ on which $I_\ell$ acts trivially. Let $k_L$ be the residue field of $L$ and let $x \mapsto \tilde{x}$ be the reduction modulo $\mathfrak{P}$ map from $\mathcal{O}_L$ to $k_L$. The following proposition of Greenberg--Vatsal explains how to compute $\tau_{\eta}(\bA_{f,\psi}/\Q_\infty)$ (cf. \cite[Proposition 2.4]{GreenbergVatsal}).
\begin{proposition}
 Let $\ell \in \Sigma_0$ and write 
 $$P_{\ell,f}(X)=\det(1-\Frob_{\ell}X|_{(\bV_{f , \psi})_{I_\ell}}) \in \mathcal{O}_L[X].$$
 Let $d_{\ell,f}$ denote the multiplicity of $X=\tilde{\ell}^{-1}$ as a root of $\tilde{P}_{\ell,f} \in k_L[X].$ Then for each prime $\eta \mid \ell$, we have $\tau_{\eta}(\bA_{f,\psi}/\Q_\infty)=d_{\ell,f}$
\end{proposition}
\begin{corollary}
    For each $\ell \in \Sigma_0$, we have $\delta_\ell(\bA_{f,\psi}/\Q_\infty)\equiv d_{\ell,f} \pmod 2.$
\end{corollary}
\begin{proof}
    Let $s_\ell$ denote the number of primes $\eta$ of $\Q_\infty$ lying over $\ell$. That is, $s_\ell=[\Gamma:\Gamma_\ell]$ where $\Gamma_\ell$ denotes the decomposition subgroup of $\Gamma$ for any such $\eta$. It follows from \cite[page 37 and Proposition 2.4]{GreenbergVatsal} that $\delta_\ell(\bA_{f,\psi}/\Q_\infty)=s_\ell d_{\ell,f}$. Since $s_\ell$ is a power of $p$, it is necessarily odd and hence the result follows.
\end{proof}
Our next goal is to compute the parity of $d_{\ell,f}$ for each prime $\ell \in \Sigma_0$.
\begin{lemma}\label{lem:noN}
    If $ \ell \nmid N$, then $d_{\ell,f}$ is odd if and only if $\tilde{\psi}$ is unramified at $\ell$ and any of the following mutually disjoint conditions \eqref{CaseI}, \eqref{CaseII}, \eqref{CaseIII} given below  hold. 
\end{lemma}
\begin{proof}
If $\ell \nmid N$ but $\tilde{\psi}$ is ramified at $\ell$ then 
$(\tilde{\bV}_{f, \psi})_{I_\ell} =0$. This can be deduced easily from \cite[last paragraph of page 255]{DiamondTaylor}.

Suppose $\ell \nmid N$ and $\tilde{\psi}$ is unramified at $\ell$, then $\tilde{\bV}_{f, \psi}$ is unramified at $\ell$ and hence $(\tilde{\bV}_{f, \psi})_{I_\ell} = \tilde{\bV}_{f, \psi}.$ It is well-known that (cf. \cite[Section 2.1]{LZ16})
    $$P_{\ell,f}(X)\equiv (1-\alpha_\ell^2\psi(\ell)X)(1-\beta_\ell^2\psi(\ell)X)(1-\alpha_\ell\beta_\ell \psi(\ell) X) \pmod{\mathfrak{P}}$$
    where $\alpha_\ell^{-1}$ and $\beta_\ell^{-1}$ are two roots of the Hecke polynomial $1-a_\ell(f)X +\varepsilon(l)\ell^{k-1}X^2.$
    Therefore, 
    \begin{align*}
       \tilde{P}_{\ell,f}(X)&=\Big(1-\psi(\ell)(\alpha_{\ell}^2+\beta_\ell^2)X+\psi(\ell)^2\alpha_\ell^2\beta_\ell^2X^2\Big)^{\sim} \Big(1-\varepsilon(\ell)\psi(\ell)\ell^{k-1}X\Big)^{\sim}\\
       &=\Big(1-\big(\psi(\ell)a_\ell(f)^2-2\psi(\ell)\varepsilon(\ell)\ell^{k-1}\big)X+\psi(\ell)^2\varepsilon(\ell)^2\ell^{2k-2}X^2\Big)^{\sim}\Big(1-\varepsilon(\ell)\psi(\ell)\ell^{k-1}X\Big)^{\sim}.
    \end{align*}
    Let $\tilde{g}(X)=\Big(1-\varepsilon(\ell)\psi(\ell)\ell^{k-1}X\Big)^{\sim}$ and $\tilde{h}(X)=\tilde{P}_{\ell,f}(X)/\tilde{g}(X)$.
    It follows that $d_{\ell,f}=1$ if and only if 
either of the following two cases hold.
\begin{enumerate}[(I)]
    \item $\tilde{\ell}^{-1}$ is a root of $\tilde{g}(X)$ and $\tilde{\ell}^{-1}$ is not a root of $\tilde{h}(X)$,
    \item $\tilde{\ell}^{-1}$ is a simple root of $\tilde{h}(X)$ and $\tilde{\ell}^{-1}$ is not a  root of $\tilde{g}(X).$
\end{enumerate}
\noindent{\textit{Case (I):}} $\tilde{\ell}^{-1}$ is a root of $\tilde{g}(X)$ if and only if 
\begin{equation}\label{eq:CaseIi}
  \varepsilon(\ell)\psi(\ell)\ell^{k-2}\equiv 1 \pmod{\mathfrak{P}}.
\end{equation}
 Also, $\tilde{\ell}^{-1}$ is not a root of $\tilde{h}(X)$ means $\tilde{h}(\tilde{\ell}^{-1}) \neq 0.$ Hence Case (I) holds if and only if 
\begin{equation}\label{CaseI}
 \text{equation } \eqref{eq:CaseIi}  \text{ holds and }  h(\ell^{-1}) \not\equiv  0 \pmod{\mathfrak{P}}.
\end{equation}

\noindent{\textit{Case (II):}} We first find equivalent conditions when $\tilde{\ell}^{-1}$ is a simple root of $\tilde{h}(X)$. Since the product of two roots of $h(X)$ is $\psi(\ell)^{-2}\varepsilon(\ell)^{-2}\ell^{2-2k}$, 

    $$\tilde{\ell}^{-1} \text{ is a  root if and only if } \Big(\psi(\ell)^{-2}\varepsilon(\ell)^{-2}\ell^{3-2k}\Big)^{\sim} \text{ is a root. }$$

Therefore, $\tilde{\ell}^{-1}$ is a root of $\tilde{h}(X)$ if and only if 
$$\ell^{-1} + \psi(\ell)^{-2}\varepsilon(\ell)^{-2}\ell^{3-2k} \equiv \big(\psi(\ell)a_\ell(f)^2-2\psi(\ell)\varepsilon(\ell)\ell^{k-1}\big)\big(\psi(\ell)^{-2}\varepsilon(\ell)^{-2}\ell^{2-2k}\big) \pmod{\mathfrak{P}}.$$ Simplifying, we get that $\tilde{\ell}^{-1}$ is a root of $\tilde{h}(X)$ if and only if  
\begin{equation}\label{eq: IIi}
   \psi(\ell)^{-1}a_\ell(f)^2\varepsilon(\ell)^{-2}\ell^{3-2k}-2\psi(\ell)^{-1}\varepsilon(\ell)^{-1}\ell^{2-k}-\psi(\ell)^{-2}\varepsilon(\ell)^{-2}\ell^{4-2k} \equiv 1 \pmod{\mathfrak{P}}.
\end{equation}
Hence, if $\tilde{l}^{-1}$ is a root of $\tilde{h}(X)$, it is a simple root  if and only if 
\begin{equation}\label{eq: IIii}
    \ell^{-1}\not\equiv \psi(\ell)^{-2}\varepsilon(\ell)^{-2}\ell^{3-2k} \pmod{\mathfrak{P}} \text{ i.e. } \psi(\ell)^{-2}\varepsilon(\ell)^{-2}\ell^{4-2k} \not\equiv 1 \pmod{\mathfrak{P}}.
\end{equation}
Therefore, Case (II) holds if and only if 
\begin{equation}\label{CaseII}
   \text{equations } \eqref{eq: IIi}  \text{  and } \eqref{eq: IIii} \text{ hold  and equation } \eqref{eq:CaseIi} \text{ does not hold.}
\end{equation}
This completes the case when $d_{\ell,f}=1$. Now, it is easy to see that $d_{\ell,f}=3$ if and only if 
\begin{equation}\label{CaseIII}
    \text{equations } \eqref{eq: IIi}  \text{  and } \eqref{eq:CaseIi} \text{ hold  and equation } \eqref{eq: IIii} \text{ does not hold.}
\end{equation}
\end{proof}

We will  need to assume the following hypothesis as in \cite[Proposition 3.3 or Theorem 3.1]{JhaMajumdarShekhar}.

\vspace{.2cm}

\noindent{\textbf{(Cond)}} $N$ is square-free and $\mathrm{cond}_\ell(\tilde{\rho}_f)=\ell$ for all primes $\ell \mid \mid N$.

\vspace{.2cm}
Two cases can arise.
\begin{itemize}
    \item $\ell \mid \mid N$ and $\ell \nmid M$; or
    \item $\ord_\ell(N)=\ord_\ell(M)=1$.
\end{itemize}
Here $M$ is the conductor of the nebentypus $\varepsilon$.

\begin{lemma}\label{lem: ii}
Suppose $\ell \mid \mid N$ and $\ell \nmid M$. Then $d_{\ell,f}$ is odd if and only if $\tilde{\psi}$ is unramified at $\ell$ and $\ell \equiv a_\ell(f)^2\psi(\ell) \pmod{\mathfrak{P}}.$
\end{lemma}
 \begin{proof}
In this case  \cite[Theorem 3.26, 3(b)]{HidabookModular} gives
$$\tilde{\rho}_f|_{G_\ell}\sim \begin{pmatrix}
\tilde{\omega_p}\tilde{\chi} & \tilde{D}\\
~ & \tilde{\chi}
\end{pmatrix},$$
where $\tilde{\chi}$ is unramified such that $\tilde{\chi}(\Frob_\ell)=\tilde{a}_\ell(f)$ . The conductor hypothesis  
$\mathrm{cond}_\ell(\tilde{\rho}_f)=\ell$ imply that $\dim \tilde{\rho}_f^{I_\ell}=1$ (cf. same argument as in the proof of Case 3 of Proposition 3.3 of \cite{JhaMajumdarShekhar}). This means that 
$\tilde{D}(g) \neq 0$  for some $g \in I_\ell$. Therefore,
$$\tilde{\bV}_f|_{G_\ell}\sim \begin{pmatrix}
\tilde{\omega_p}^2\tilde{\chi}^2 & \tilde{\omega_p}\tilde{\chi}\tilde{D} & \tilde{D}^2\\
~ & \tilde{\omega_p}\tilde{\chi}^2 & 2\tilde{\chi}\tilde{D}\\
~ & ~ & \tilde{\chi}^2
\end{pmatrix} 
\text{ and hence }
\tilde{\bV}_{f,\psi}|_{G_\ell}\sim \begin{pmatrix}
\tilde{\omega_p}^2\tilde{\chi}^2\tilde{\psi} & \tilde{\omega_p}\tilde{\chi}\tilde{D}\tilde{\psi} & \tilde{D}^2\tilde{\psi}\\
~ & \tilde{\omega_p}\tilde{\chi}^2\tilde{\psi} & 2\tilde{\chi}\tilde{D}\tilde{\psi}\\
~ & ~ & \tilde{\chi}^2\tilde{\psi}
\end{pmatrix}.$$
If $\tilde{\psi}$ is unramified at $\ell$, then
the action of $I_\ell$ on on $\tilde{\bV}_{f,\psi}$ is via the matrix 
$\begin{pmatrix}
1 & \tilde{D} & \tilde{D}^2\\
~ & 1 & 2\tilde{D}\\
~ & ~ & 1
\end{pmatrix}$
and hence $(\tilde{\bV}_{f,\psi})_{I_\ell}$ is one dimensional and the action of $\Frob_\ell$ on this space is via $\tilde{\chi}^2\tilde{\psi}$. Thus $\tilde{P}_{\ell,f}=(1-\tilde{a}_\ell(f)^2\psi(\ell)X)$. It follows that $\tilde{\ell}^{-1}$ is a root of $\tilde{P}_{\ell,f}$ if and only if $\ell \equiv a_\ell(f)^2\psi(\ell) \pmod{\mathfrak{P}}.$

If $\tilde{\psi}$ is ramified at $\ell$, then
the action of $I_\ell$ on on $\tilde{\bV}_{f,\psi}$ is via the matrix 
$\begin{pmatrix}
\tilde{\psi} & \tilde{D}\tilde{\psi} & \tilde{D}^2\tilde{\psi}\\
~ & \tilde{\psi} & 2\tilde{D}\tilde{\psi}\\
~ & ~ & \tilde{\psi}
\end{pmatrix}$
and in this case $(\tilde{\bV}_{f,\psi})_{I_\ell}=0.$
\end{proof}
Next we deal with the case $\ord_\ell(N)=\ord_\ell(M)=1$.
In this case,  $$\tilde{\rho}_f|_{G_\ell}\sim \begin{pmatrix}
\tilde{\chi}_1 & ~ \\
~ & \tilde{\chi}_2
\end{pmatrix}$$
where $\tilde{\chi}_2$ is an unramified character such that $\tilde{\chi}_2(\Frob_\ell)=\tilde{a}_\ell(f)$ (cf. \cite[Theorem 3.26(3a)]{HidabookModular}). As $N$ is square-free and $\mathrm{ord}_\ell(N)>0$, this means that $\ell \mid \mid N$ and hence we also have that 
$\mathrm{cond}_\ell(\tilde{\rho}_f)=\ell $ by our hypothesis \textbf{(Cond)}. By the same argument as before, this means $\dim \tilde{\rho}_f^{I_\ell}=1$ 
and hence the character $\tilde{\chi}_1$ must be 
ramified.
It follows that 
$$\tilde{\bV}_f|_{G_\ell}\sim \begin{pmatrix}
\tilde{\chi}_1^2 & ~ & ~\\
~ & \tilde{\chi}_1\tilde{\chi}_2 & ~\\
~ & ~ & \tilde{\chi}_2^2
\end{pmatrix}
\text{ and hence }
\tilde{\bV}_{f,\psi}|_{G_\ell}\sim \begin{pmatrix}
\tilde{\chi}_1^2\tilde{\psi} & ~ & ~\\
~ & \tilde{\chi}_1\tilde{\chi}_2\tilde{\psi} & ~\\
~ & ~ & \tilde{\chi}_2^2\tilde{\psi}
\end{pmatrix}.
$$
\begin{lemma}\label{lem: iii}
Suppose $\ord_\ell(N)=\ord_\ell(M)=1$. The $d_{\ell,f}$ is odd if and only if 
\begin{itemize}
      \item  $\varepsilon(\ell)\ell^{k-2}\psi(\ell)\equiv 1 \pmod{\mathfrak{P}}$ if $\tilde{\chi}_1\tilde{\psi}$ is unramified.
      \item $a_\ell(f)^2\psi(\ell) \ell^{-1} \equiv 1 \pmod{\mathfrak{P}}$ if $\tilde{\chi}_1\tilde{\psi}$ and $\tilde{\chi}_1^2\tilde{\psi}$ are both ramified.
      \item ${a}_\ell(f)^{-2}{\ell}^{2k-3}{\varepsilon}(\ell)^2{\psi}(\ell) \equiv 1 \pmod{\mathfrak{P}}$ if $\tilde{\chi}_1\tilde{\psi}$ and $\tilde{\psi}$ are ramified  but $\tilde{\chi}_1^2\tilde{\psi}$ is unramified.
      \item any one of equations \eqref{casejii} or \eqref{casejiii} below holds and the other does not hold, if $\tilde{\chi}_1\tilde{\psi}$ is ramified  but $\tilde{\chi}_1^2\tilde{\psi}$ and $\tilde{\psi}$ are unramified.
\end{itemize}
The bullets above exhaust all the possibilities that can occur when $\ord_\ell(N)=\ord_\ell(M)>0$.
\end{lemma}
\begin{proof} 
\noindent{\textit{Case 1:}} Suppose that the character $\tilde{\chi}_1\tilde{\psi}$ is unramified. Since $\tilde{\chi}_1 $ is ramified, it implies that $\tilde{\chi}_1^2\tilde{\psi}$ is also ramified. The character $\tilde{\psi}$ must also be ramified at $\ell$ because if not then $\tilde{\chi_1}=(\tilde{\chi}_1\tilde{\psi})(\tilde{\psi}^{-1})$ becomes unramified which is a contradiction. 
Then 
$\tilde{\bV}_{f,\psi}|_{I_\ell}\sim \begin{pmatrix}
\tilde{\chi}_1^2\tilde{\psi} & ~ & ~\\
~ & 1 & ~\\
~ & ~ & \tilde{\psi}
\end{pmatrix}.$
It follows that the the space $(\tilde{\bV}_{f,\psi})_{I_\ell}$ is one dimensional and the action of $\Frob_\ell$ on this space is via $\tilde{\chi}_1\tilde{\chi}_2\tilde{\psi}$. Therefore $\tilde{P}_{\ell,f}=(1-\tilde{\varepsilon}(\ell)\tilde{\ell}^{k-1}\tilde{\psi}(\ell)X)$. It follows that $d_{\ell,f}=1$ if and only if 
\begin{equation}\label{casej1}
\varepsilon(\ell)\ell^{k-2}\psi(\ell)\equiv 1 \pmod{\mathfrak{P}}.
\end{equation}

\noindent{\textit{Case 2:}} Suppose that both the characters $\tilde{\chi}_1\tilde{\psi}$ and $\tilde{\chi}_1^2\tilde{\psi}$ are ramified.
Then $(\tilde{\bV}_{f,\psi})_{I_\ell}$ is nontrivial if and only if $\tilde{\psi}$ is unramified at $\ell$; in this case the action of $\Frob_\ell$ on the one-dimensional space $(\tilde{\bV}_{f,\psi})_{I_\ell}$ is given by $\tilde{\chi}_2^2\tilde{\psi}$. Therefore, $\tilde{P}_{\ell,f}=(1-\tilde{a}_\ell(f)^2\tilde{\psi}(\ell)X)$. It follows that $d_{\ell,f}=1$ if and only if 
\begin{equation}\label{casejii}
a_\ell(f)^2\psi(\ell) \ell^{-1} \equiv 1 \pmod{\mathfrak{P}}.
\end{equation}
\noindent{\textit{Case 3:}} Suppose that the character $\tilde{\chi}_1\tilde{\psi}$ is ramified but $\tilde{\chi}_1^2\tilde{\psi}$ is unramified.
Now there are two subcases of this, which we deal separately.

\noindent{\textit{Subcase (3i):}} $\tilde{\psi}$ is ramified at $\ell$. In this case 
$\tilde{\bV}_{f,\psi}|_{I_\ell}\sim \begin{pmatrix}
1 & ~ & ~\\
~ & \tilde{\chi}_1\tilde{\psi} & ~\\
~ & ~ & \tilde{\psi}
\end{pmatrix}.$ It follows that the space $(\tilde{\bV}_{f,\psi})_{I_\ell}$ is one dimensional and the action of $\Frob_\ell$ on this space is via the character $\tilde{\chi}_1^2\tilde{\psi}=\det(\tilde{\rho}_f)^2\tilde{\chi}_2^{-2}\tilde{\psi}$. It follows that $\tilde{P}_{\ell,f}=\big(1-\tilde{a}_\ell(f)^{-2}\tilde{\ell}^{2k-2}\tilde{\varepsilon}(\ell)^2\tilde{\psi}(\ell)X\big)$. Do $d_{\ell,f}=1$ if and only if 
\begin{equation}\label{casejiii}
{a}_\ell(f)^{-2}{\ell}^{2k-3}{\varepsilon}(\ell)^2{\psi}(\ell) \equiv 1 \pmod{\mathfrak{P}}.
\end{equation}

\noindent{\textit{Subcase (3ii):}} $\tilde{\psi}$ is unramified at $\ell$. This means that $\tilde{\chi}_1$ is ramified since we are under the assumption that $\tilde{\chi}_1\tilde{\psi}$ is ramified.
In this case 
$\tilde{\bV}_{f,\psi}|_{I_\ell}\sim \begin{pmatrix}
1 & ~ & ~\\
~ & \tilde{\chi}_1 & ~\\
~ & ~ & 1
\end{pmatrix}.$ 
It follows that the the space $(\tilde{\bV}_{f,\psi})_{I_\ell}$ is two dimensional and the action of $\Frob_\ell$ on this space is via the diagonal matrix $\begin{pmatrix}
\tilde{\chi}_1^2\tilde{\psi} & ~\\
~ & \tilde{\chi}_2^2\tilde{\psi}
\end{pmatrix}.$
Therefore $\tilde{P}_{\ell,f}=\big(1-\tilde{a}_\ell(f)^{-2}\tilde{\ell}^{2k-2}\tilde{\varepsilon}(\ell)^2\tilde{\psi}(\ell)X\big)\big(1-\tilde{a}_\ell(f)^2\tilde{\psi}(\ell)X\big)$. Hence, 
$d_{\ell,f}=1$ if and only if any one of equations \eqref{casejii} or \eqref{casejiii} holds and the other does not hold.
\end{proof}
Summarizing, we have shown the following proposition.
\begin{proposition}\label{prop: mod2}
   We define $\mathcal{S}_{f, \psi}\subset \Sigma_0$  to be the subset consisting of primes $\ell$ satisfying Lemma \ref{lem:noN}, Lemma \ref{lem: ii} and Lemma \ref{lem: iii} such that $d_{\ell,f}$ is odd. Then $$\sum_{\ell \in \Sigma_0} \delta_\ell(\bA_{f,\psi}/\Q_\infty)\equiv |\mathcal{S}_{f, \psi}| \pmod 2.$$
\end{proposition}
\section{Congruent modular forms}\label{sec:congruent}
We consider two modular forms $f_i$ (for $i=1,2$) of level $N_i$ and character $\varepsilon_i$ as in section \ref{basics}. By enlarging $L$, if necessary, we assume that $a_n(f_i) \in L$ for all $n$. Similarly, enlarging $\Sigma$ if necessary, we assume that $\Sigma$ is a set of places of $\Q$ that contains $p, \infty$, the primes dividing $N_1N_2$ and  the primes dividing the conductor $N_\psi$. We continue to assume that \textbf{(Hyp)} is true for both $f_1$ and $f_2$. We further assume 
that the residual representations are isomorphic, i.e. 
$$\tilde{\rho_1}\cong \tilde{\rho_2}.$$ Under the above circumstances the following result is a work of Ray--Sujatha--Vatsal (cf. \cite[Proposition 3.11]{RaySujathaVatsal}).
\begin{proposition}\label{prop:Anwesh}
The $\mu$-invariant of $\Sel_{p^\infty}(\bA_{f_1, \psi}/\Q_\infty)$ vanishes if and only if the $\mu$-invariant of $\Sel_{p^\infty}(\bA_{f_2, \psi}/\Q_\infty)$ vanishes. Moreover, if these $\mu$-invariants are zero, then the imprimitive-$\lambda$-invariants coincide, i.e.
$$\lambda^{\Sigma_0}(\bA_{f_1, \psi}/\Q_\infty)=\lambda^{\Sigma_0}(\bA_{f_2, \psi}/\Q_\infty).$$
\end{proposition}
\begin{theorem}\label{Mainthm} Assume the vanishing of the $\mu$-invariants as in proposition \ref{prop:Anwesh}.
    We have the congruence
    $$\lambda(\bA_{f_1, \psi}/\Q_\infty) +|\mathcal{S}_{f_1,\psi}|\equiv \lambda(\bA_{f_2, \psi}/\Q_\infty)+|\mathcal{S}_{f_2,\psi}| \pmod{2}.$$
\end{theorem}
\begin{proof}
     Proposition \ref{prop:Anwesh} and \eqref{eq: lambda} give $$\lambda(\bA_{f_1,\psi}/\Q_\infty) +\sum_{\ell \in \Sigma_0}\delta_\ell(\bA_{f_1,\psi}/\Q_\infty)=\lambda(\bA_{f_2,\psi}/\Q_\infty) +\sum_{\ell \in \Sigma_0}\delta_\ell(\bA_{f_2,\psi}/\Q_\infty).$$
     The result now follows from proposition \ref{prop: mod2}.
\end{proof}
\begin{remark}\label{rem:obs}

The natural generalization of formula \eqref{Green} for $\bV_{f_i,\psi}$ cannot be deduced easily since one would need to apply results of \cite{Flach}. But Flach's results  (in particular \cite[Theorem 2]{Flach}) further rely on the representation being self dual i.e. $V \cong V^\vee(1)$ as in \cite[paragraph after equation 18 in section ``Antisymmetry"]{Flach}).
This is true when $V$ arises from an elliptic curve by Weil pairing. But this relation is not true for the symmetric square representation associated to an elliptic curve. The relation would be 
$\mathrm{Sym}^2(V) \cong \mathrm{Sym}^2(V^\vee(1))=(\mathrm{Sym}^2(V)^\vee)(2).$
Further, in our article we are taking twists by non-trivial Dirichlet character. Suppose $V$ satisfies $V\cong V^\vee(1)$. Then for $V(\psi)$ to be self-dual in this sense, one would need $\psi$ to be quadratic character. But in the hypotheses of \cite[Theorem C]{LZ16} quadratic characters are omitted. For all these reasons  Flach's result  cannot be applied to $\bV_{f_i,\psi}$.
Hence the results in this article are only restricted to explicitly calculating the parity of the error terms in the comparison of $\lambda$-invariants of congruent representations. We will come back to this question in the future.

\end{remark}
\section{The non-ordinary case}
We recall the setup as in \cite{BLV} with some notational changes to match with section \ref{sec:notation}. Let $f$ is a normalized, cuspidal, eigen-newform of weight $k$ (in \cite{BLV} it is $k+2$), level $N$ and nebentypus $\varepsilon$. We also assume that $p \nmid N$ and $p\geq 2k+1$ is an odd prime such that $a_p(f)=0$. As before let $\Sigma$ be  a finite set of places of $\Q$ that contains $p$, primes that divide the level $N$, primes that divide the conductor $N_\psi$ and  $\infty$. We write $\pm\alpha$ for the roots of the Hecke polynomial $X^2+ \varepsilon(p)p^{k-1}$ of $f$ at $p$. Let $L/\Q$ be a number field containing the Hecke field $\Q(\{a_n(f)\}_{n \geq 1})$ of $f$ as well as $\alpha^2$ and the image of a Dirichlet character $\psi$ of conductor $N_\psi$ coprime to $Np$. Assume that  $\psi$ satisfies all the conditions mentioned in \cite[page 3]{BLV} (our $\psi$ is their $\chi^{-1}$)\footnote{We have assumed this only to ensure that the theory in \cite{BLV} applies, in particular, the definition of signed Selmer groups and their cotorsioness conjecture (see Conjecture \ref{conj:torsion}). If we assume Conjecture \ref{conj:torsion}, then we don't need any further specific properties of $\psi$ for our results.}. Let $\mathfrak{P}$ be a prime of $L$ above $p$ and let $\O$ be the ring of integers of the completion $L_{\mathfrak{P}}$. Let us put $V_f^*=\Hom(V_f, L_{\mathfrak{P}})$ and we endow it with the contragredient Galois action. We set $\bM_{f, \psi^{-1}}:=\mathrm{Sym}^2 \text{ }T_f^*(1+\psi^{-1}):=(\mathrm{Sym}^2 \text{ }T_f^*)(1)(\psi^{-1})$. Let $\Gamma_0=\Gal(\Q_p(\mu_{p^\infty})/\Q)$ so that $ \Gamma_0 = \Delta \times \Gamma$ where $\Delta$ is a finite group of order $p-1$ and $\Gamma \cong \Z_p$. The hypothesis $a_p(f)=0$ gives a $G_{\Q_p}$-equivariant decomposition
$$\bM_{f, \psi^{-1}}=M_{1,f, \psi^{-1}} \oplus M_{2,f, \psi^{-1}}$$
(cf. \cite[page 5 and page 25]{BLV}) and exploiting this decomposition, Büyükboduk--Lei--Venkat defined three signed Coleman maps $$\col^\clubsuit:\HIw(\Q_p(\mu_{p^\infty}),\bM_{f, \psi^{-1}}) \rightarrow \Lambda_{\O}(\Gamma_0)$$
for $\clubsuit \in \{+,-,\bullet\}$ (see \cite[Section 4.2]{BLV}). 
The kernels of these maps are used to define certain local Selmer conditions at $p$ which leads to the following definition of doubly signed Selmer groups (cf. \cite[Defn. 4.4.1]{BLV}).
Let us set $\bM_{f, \psi^{-1}}^\vee(1):=(\bM_{f, \psi^{-1}})^\vee(1)$.
\begin{definition}\label{def}
    Let $\mathcal{S}$ denote the set of pairs $\{(+,-),(+,\bullet), (-,\bullet)\}.$ For $\fS = (\clubsuit, \spadesuit) \in \fS$, we define the discrete Selmer group $\Sel_{\fS}(\bM_{f, \psi^{-1}}^\vee(1)/\Q(\mu_{p^\infty}))$ as the kernel of the restriction map
    \begin{align*}
       H^1(\Q(\mu_{p^\infty}), \bM_{f, \psi^{-1}}^\vee(1))\rightarrow \prod_{v \mid p}\frac{ H^1(\Q(\mu_{p^\infty})_v, \bM_{f, \psi^{-1}}^\vee(1))}{ H^1_{\fS}(\Q(\mu_{p^\infty})_v, \bM_{f, \psi^{-1}}^\vee(1))} \times \prod_{v \nmid p}\frac{H^1(\Q(\mu_{p^\infty})_v, \bM_{f, \psi^{-1}}^\vee(1))}{H^1_{\mathrm{un}}(\Q(\mu_{p^\infty})_v, \bM_{f, \psi^{-1}}^\vee(1))},
    \end{align*}
    where $v$ runs through all primes of $\Q(\mu_{p^\infty})$ and for $v \mid p$, the local condition $H^1_{\fS}(\Q(\mu_{p^\infty})_v, \bM_{f, \psi^{-1}}^\vee(1))$ is the orthogonal complement of $\ker\big(\col^\clubsuit\big) \cap \ker\big(\col^\spadesuit\big)$ under the local Tate pairing.
\end{definition}

\begin{remark}
   We have taken $\psi^{-1}$ in the definition of $\bM_{f, \psi^{-1}}$ because $\bM_{f, \psi^{-1}}^\vee(1)=\bT_{f, \psi} \otimes \Q_p/\Z_p=\bA_{f, \psi}$ (cf. \cite[Notation 3.2.4]{LZ16}) which coincides with the notation we fixed in section \ref{sec:notation}.
\end{remark}
Here is a conjecture on the cotorsioness of these Selmer groups made in \cite[Conjecture 4.4.3]{BLV}
\begin{conjecture}\label{conj:torsion}
For every $\fS \in \mathcal{S}$, and every character $\eta$ of $\Delta$, the $\eta$-isotypic component $e_\eta\Sel_{\fS}(\bM_{f, \psi^{-1}}^\vee(1)/\Q(\mu_{p^\infty}))$ is $\Lambda$-cotorsion.
\end{conjecture}
Some evidence for this conjecture is also provided (see \cite[Theorem B, (ii)]{BLV}).

\subsection{The cyclotomic and finite level.}  Recall that $\Gamma$ is the Galois group of the cyclotomic extension $\Q_\infty$ over $\Q$ and $\Gamma_n=\Gal(\Q_\infty/\Q_{(n)})$ where $\Q_{(n)}$ is the extension over $\Q$  such that $[\Q_{(n)}:\Q]=p^n$. For $v \mid p$, we set $$H^1_{\fS}((\Q_\infty)_v, \bM_{f, \psi^{-1}}^\vee(1)):=H^1_{\fS}(\Q(\mu_{p^\infty})_v, \bM_{f, \psi^{-1}}^\vee(1))^\Delta$$ and $$H^1_{\fS}((\Q_{(n)})_v, \bM_{f, \psi^{-1}}^\vee(1)):=H^1_{\fS}((\Q_\infty)_v, \bM_{f, \psi^{-1}}^\vee(1))^{\Gamma_n}$$
and we define the corresponding Selmer groups $\Sel_{\fS}(\bM_{f, \psi^{-1}}^\vee(1)/\Q_\infty)$ and $\Sel_{\fS}(\bM_{f, \psi^{-1}}^\vee(1)/\Q_{(n)})$ with these local conditions. 
\begin{theorem}\label{prop:finite}
    The Selmer group $\Sel_{\fS}(\bA_{f, \psi}/\Q_\infty)$ contains no proper $\Lambda$-submodule of finite index.
\end{theorem}
Before we prove this theorem we have to prove some preliminary lemmas. 
\vspace{.2cm}

For $s \in \Z$, we can take the cyclotomic twist $\bA_{f, \psi,s}:=\bA_{f, \psi} \otimes (\omega|_{\Gamma})^s$ where $\omega|_{\Gamma}:\Gamma \rightarrow 1+p\Z_p$ is an isomorphism and one can define the corresponding Selmer group $\Sel_{\fS}(\bA_{f, \psi,s}/\Q_\infty)$  just as before. For the prime $v$ above $p$ one uses the $G_{\Q_p}$-invariant submodule $$H^1_{\fS}((\Q_\infty)_v, \bA_{f, \psi, s})):=H^1_{\fS}((\Q_\infty)_v, \bA_{f, \psi}))\otimes (\omega|_{\Gamma})^s.$$
As a $\Gal(\overline{\Q}/\Q_\infty)$-module $\bA_{f, \psi, s} \cong \bA_{f, \psi}$ and thus $H^1(\Q_\infty, \bA_{f, \psi, s}) \cong H^1(\Q_\infty, \bA_{f, \psi}) \otimes (\omega|_{\Gamma})^s $. For a prime $v$, $H^1((\Q_\infty)_v, \bA_{f, \psi, s}) \cong H^1((\Q_\infty)_v, \bA_{f, \psi}) \otimes (\omega|_{\Gamma})^s$. For the finite level $\Q_{(n)}$, we can similarly define the twisted Selmer group $\Sel_{\fS}(\bA_{f, \psi,s}/\Q_{(n)})$ as above with the local condition at $p$ defined as $H^1((\Q_\infty)_v, \bA_{f, \psi, s})^{\Gamma_n}$. Thus we remark that for $K=\Q_\infty$, $\Sel_{\fS}(\bA_{f, \psi,s}/K)\cong \Sel_{\fS}(\bA_{f, \psi}/K)\otimes (\omega|_{\Gamma})^s$ as $\Lambda$-modules. 
\vspace{.2cm}

Let $\bM_{f, \psi^{-1}, -s}:=\bM_{f, \psi^{-1}} \otimes (\omega|_{\Gamma})^{-s}$. Let $\bJ_{f,\psi^{-1},-s}=\bM_{f, \psi^{-1}, -s} \otimes \Q_p/\Z_p $. Using the signed Coleman maps 
$\col^\clubsuit:\HIw(\Q_p(\mu_{p^\infty}),\bT_{f, \psi}) \rightarrow \Lambda_{\O}(\Gamma_0)$ one can similarly define $\Sel_{\mathfrak{S}}(\mathbf{J}_{f,\psi^{-1}}/\Q_\infty)$ and also the twisted versions $\Sel_{\fS}(\bJ_{f,\psi^{-1},-s}/\Q_\infty)$ and $\Sel_{\fS}(\bJ_{f,\psi^{-1},-s}/\Q_{(n)} )$ as before.
We begin with a "control theorem" for these signed Selmer groups.

For the following two lemmas we will need the following hypothesis (as in \cite[Section 1.2]{ponsinet} or \cite[last line in p. 1265]{HatleyLei19}).

\begin{enumerate}[~] 
    \item \textbf{(invS)} The cohomology groups  $H^0(\Q_p,\bJ_{f,\psi^{-1}})$ and $H^0(\Q_p,\bA_{f,\psi})$ are trivial.
\end{enumerate}
\begin{lemma}\label{lem1}
    For all but finitely many $s \in \Z$, the kernel and cokernel of the restriction map 
    $$\Sel_{\fS}(\bJ_{f,\psi^{-1},-s}/\Q_{(n)} )\rightarrow \Sel_{\fS}(\bJ_{f,\psi^{-1},-s}/\Q_\infty)^{\Gamma_n}$$ are finite of bounded orders as $n$ varies.
\end{lemma}
\begin{proof}
    Consider the commutative diagram 
    
	\begin{equation}\label{Selmer_definition}
	\begin{tikzcd}
	0 \arrow[r] & \Sel_{\fS}(\bJ_{f,\psi^{-1},-s}/\Q_{(n)} ) \arrow[d, "\rho_{{n}}"] \arrow[r] &  H^1(\Q_{(n)}, \bJ_{f,\psi^{-1},-s}) \arrow[d, "h"]  \arrow[r, "\lambda"] & \bigoplus_{v}\mathcal{P}_{v}(\bJ_{f,\psi^{-1},-s}/\Q_{(n)} ) \arrow[d, "\oplus_v\Xi_{{n,v}}"] \\
		0 \arrow[r] & \Sel_{\fS}(\bJ_{f,\psi^{-1},-s}/\Q_\infty)^{\Gamma_n} \arrow[r] & H^1(\Q_{\infty}, \bJ_{f,\psi^{-1},-s})^{\Gamma_n}  \arrow[r] & \bigoplus_{v}\mathcal{P}_{v}(\bJ_{f,\psi^{-1},-s}/\Q_{\infty})^{\Gamma_n}
	\end{tikzcd}
	\end{equation}
Here $\mathcal{P}_v(\bJ_{f,\psi^{-1},-s},-)$ is the local term at place $v$ defining the corresponding signed Selmer group. By hypothesis \textbf{(invS)} and the inflation restriction exact sequence the map $h$ is an isomorphism. Next we analyse the kernel and cokernel of the map $\Xi_{n,v}$. 

For the prime $v$ above $p$, consider the commutative diagram 
\begin{equation}\label{local}
	\begin{tikzcd}
	0 \arrow[r] & H^1_{\fS}((\Q_{(n)})_v, \bJ_{f,\psi^{-1},-s})) \arrow[d, "m"] \arrow[r] &  H^1((\Q_{(n)})_v,\bJ_{f,\psi^{-1},-s})) \arrow[d, "g"]  \arrow[r, twoheadrightarrow] & \frac{H^1((\Q_{(n)})_v, \bJ_{f,\psi^{-1},-s}))}{H^1_{\fS}((\Q_{(n)})_v, \bJ_{f,\psi^{-1},-s}))} \arrow[d, "\Xi_{n,v}"] \\
		0 \arrow[r] & H^1_{\fS}((\Q_\infty)_v, \bJ_{f,\psi^{-1},-s}))^{\Gamma_n} \arrow[r] &  H^1((\Q_\infty)_v, \bJ_{f, \psi^{-1},- s}))^{\Gamma_n}\arrow[r] & \left(\frac{H^1((\Q_\infty)_v, \bJ_{f,\psi^{-1},-s}))}{H^1_{\fS}((\Q_\infty)_v,\bJ_{f,\psi^{-1},-s}))}  \right)^{\Gamma_n}
	\end{tikzcd}
	\end{equation}
 The map $m$ is an isomorphism by definition. The central map $g$ is an isomorphism by inflation restriction exact sequence and \textbf{(invS)}. Therefore, the map $\Xi_{n,v}$ is injective. 
 
 For the prime $v \nmid p$, it is a usual known argument (see for example \cite[page 1645, proof of Lemma 2.3]{ponsinet} or \cite[second paragraph of page 1275]{HatleyLei19}) that the map $\Xi_{n,v}$ has finite kernel of bounded orders as $n$ varies. 
\end{proof}

\begin{lemma}\label{lem: first}
  Suppose that the Selmer group $\Sel_{\fS}(\bJ_{f, \psi^{-1}, -t}/\Q_\infty)$ is $\Lambda$-cotorsion for some fixed $t \in \Z$.    Then for all but finitely many $s$, the map 
    $$H^1(\Q,\bA_{f, \psi,s}) \rightarrow  \bigoplus_{v}\mathcal{P}_{v}(\bA_{f, \psi,s}/\Q ) $$
    is surjective, and for all $s \in \Z$, the map 
     $$H^1(\Q_\infty,\bA_{f, \psi,s}) \rightarrow  \bigoplus_{v}\mathcal{P}_{v}(\bA_{f, \psi,s}/\Q_\infty ) $$
     is surjective.
\end{lemma}
\begin{proof}
    Since the Selmer group $\Sel_{\fS}(\bJ_{f, \psi^{-1},-t}/\Q_\infty)$ is $\Lambda$-cotorsion, then for all but finitely many $u\in \Z$, $(\Sel_{\fS}(\bJ_{f, \psi^{-1},-t}/\Q_\infty) \otimes \omega^u)^{\Gamma_n}=\Sel_{\fS}(\bJ_{f, \psi^{-1},u-t}/\Q_\infty)^{\Gamma_n}$ is finite for every $n$. Since $t$ is fixed, it means that for all but finitely many $s \in \Z$, $\Sel_{\fS}(\bJ_{f,\psi^{-1},-s}/\Q_\infty)^{\Gamma_n}$ is finite for every $n$. Thus by Lemma \ref{lem1}, possibly avoiding another finite number of $s \in \Z$, $\Sel_{\fS}(\bJ_{f,\psi^{-1},-s}/\Q_{(n)})$ is finite for every $n$. For such an $s$ and any $n$, \cite[Proposition 4.13]{Greenberg} shows that the cokernel of the map 
    \begin{equation}\label{eq:mithu}
    H^1(\Q_{(n)},\bA_{f, \psi,s}) \rightarrow  \bigoplus_{v}\mathcal{P}_{v}(\bA_{f, \psi,s}/\Q_{(n)} ) 
    \end{equation}
    is the Pontryagin dual of $H^0(\Q_{(n)},\bJ_{f,\psi^{-1},-s}).$  By \textbf{(invS)}  we have $H^0(\Q, \bJ_{f, \psi^{-1},0})=0$. Since $\Gal(\Q_\infty/\Q)$ is pro-$p$,   $H^0(\Q_\infty, \bJ_{f, \psi^{-1},0})$ is trivial. Moreover, as  $\bJ_{f,\psi^{-1},-s} \cong  \bJ_{f, \psi^{-1},0}$ as $\Gal(\overline{\Q}/\Q_\infty)$-modules, hence $H^0(\Q, \bJ_{f,\psi^{-1},-s})=0$ and $H^0(\Q_{(n)}, \bJ_{f,\psi^{-1},-s})=0$.
Therefore the map in \eqref{eq:mithu} is surjective for any $n$. The lemma then follows by passing to the direct limit relative to the restriction maps and noting that, for any $s$,  $\bA_{f, \psi, s} \cong \bA_{f, \psi}$ as $\Gal(\overline{\Q}/\Q_\infty)$-modules.
\end{proof}
\begin{lemma}\label{lemma:second}
    For all $s\in \Z$, the restriction map $$\mathcal{P}_{v}(\bA_{f, \psi,s}/\Q ) \rightarrow \mathcal{P}_{v}(\bA_{f, \psi,s}/\Q_\infty)^\Gamma $$ is surjective.
\end{lemma}
\begin{proof}
Let $v$ be a prime above $p$. 
    Since $a_p(f_i)=0$, we have the decomposition
    $\bM_{f_i, \psi^{-1}} = M_{1,f_i, \psi^{-1}} \oplus M_{2,f_i, \psi^{-1}}$
    as $G_{\Q_p}$-modules where $M_{j,f_i, \psi^{-1}}$ is of rank $j$ (cf. \cite[Corollary 4.1.2]{BLV}). In \cite{BLV}, Büyükboduk--Lei--Venkat defined the Coleman maps $\col_{f_i}^\pm$ from the rank $2$ lattice $M_{2,f_i, \psi^{-1}}$ generalizing methods of \cite{leiloefflerzerbes11} and \cite{leiloefflerzerbes10} while the Coleman map $\col_{f_i}^\bullet$ was  defined using the  rank $1$ lattice $M_{1,f_i, \psi^{-1}}$ (cf. \cite[Lemma 4.1.5 and definition 4.2.1]{BLV}). More precisely, 
    they show the existence of Coleman maps
    \begin{align*}
        \overline{\col}^{\pm}_{f_i}: \HIw(\Q_p(\mu_{p^\infty}),M_{2, f_i, \psi^{-1}}) &\rightarrow \Lambda_{\O}(\Gamma_0)\\
        \overline{\col}^{\bullet}_{f_i}: \HIw(\Q_p(\mu_{p^\infty}),M_{1, f_i, \psi^{-1}}) &\rightarrow \Lambda_{\O}(\Gamma_0)
    \end{align*}
    such that $\col^?_{f_i} =  \overline{\col}^{?}_{f_i} \circ \mathrm{Proj}_{{f_i},?}. $
      Here $\mathrm{Proj}_{{f_i},?}:\HIw(\Q_p(\mu_{p^\infty}), \bM_{f, \psi^{-1}}) \rightarrow \HIw(\Q_p(\mu_{p^\infty}), M_{j,f_i, \psi^{-1}})$ is the natural projection map where $j=2$ if $?=+ \text{ or } -$, and $j=1$ if $?=\bullet$.

  Recall that $\mathcal{S}$ was the set of pairs $\{(+,-),(+,\bullet), (-,\bullet)\}.$ If $\fS = (\clubsuit, \spadesuit) \in \mathcal{S}$,
    then define
    \begin{align*}
       \col_{f_i}^{\fS}:\HIw(\Q_p(\mu_{p^\infty}), \bM_{f_i, \psi^{-1}}) &\rightarrow \Lambda_{\O}(\Gamma_0)^{\oplus 2} \\
       z &\mapsto \col_{f_i}^{\clubsuit} (z)\oplus \col_{f_i}^{\spadesuit} (z).
    \end{align*}
 
By remark \ref{rem:gen} below, the Pontryagin dual of $H^1_{\fS}((\Q(\mu_{p^\infty}))_v,\bA_{f_i, \psi})$ is isomorphic to $\im \col_{f_i}^{\fS}$ which is contained in a free $\Z_p[[\Gamma_0]]$-module.
Therefore, the Pontryagin dual of 
   $H^1_{\fS}((\Q_\infty)_v,\bA_{f, \psi})$ is contained in a free $\Z_p[[\Gamma]]$-module,
   and hence $(H^1_{\fS}(\Q_\infty)_v,\bA_{f, \psi})_{\Gamma}=0$. This   implies that  $H^1_{\fS}((\Q_\infty)_v,\bA_{f, \psi,s})_{\Gamma}$ is trivial. 
   Therefore, we have an exact sequence
   $$0 \rightarrow H^1_{\fS}((\Q_\infty)_v,\bA_{f, \psi,s})^{\Gamma} \rightarrow H^1((\Q_\infty)_v,\bA_{f, \psi,s})^{\Gamma} \rightarrow \Big(\frac{H^1((\Q_\infty)_v,\bA_{f, \psi,s})}{H^1_{\fS}((\Q_\infty)_v,\bA_{f, \psi,s})}\Big)^{\Gamma} \rightarrow 0.$$
   By \textbf{(invS)}, $H^1((\Q_\infty)_v,\bA_{f, \psi,s})^{\Gamma} \cong H^1(\Q_p,\bA_{f, \psi,s})$ and by definition $ H^1_{\fS}((\Q_\infty)_v,\bA_{f, \psi,s})^{\Gamma} \cong H^1_{\fS}(\Q_p,\bA_{f, \psi,s})$. This gives a surjection
   $$\frac{H^1(\Q_p,\bA_{f, \psi,s})}{H^1_{\fS}(\Q_p,\bA_{f, \psi,s})}\rightarrow \Big(\frac{H^1((\Q_\infty)_v,\bA_{f, \psi,s})}{H^1_{\fS}((\Q_\infty)_v,\bA_{f, \psi,s})}\Big)^{\Gamma}.$$
   For the primes not above $p$, the proof is standard (cf. \cite[Lemma 2.5]{ponsinet}).   
\end{proof}
\begin{remark}
    Please note that the proofs of Lemmas \ref{lem1}, \ref{lem: first} and \ref{lemma:second} follow exactly as in Lemmas 2.3, 2.4 and 2.5 of \cite{ponsinet} (where the setting is supersingular abelian varieties). In particular in the proof of Lemma \ref{lem: first} (analogous to \cite[Lemma 2.4]{ponsinet}) we note that we have used \cite[Proposition 4.13]{Greenberg} since, in Greenberg's paper, for this particular proposition, the local conditions defining Selmer group is quite general (it can be any divisible subgroup of $H^1(F_v,M)$ where $M \cong (\Q_p/\Z_p)^d$ for any $d \geq 1$). See page 49 of \cite{Greenberg} where this general setting is considered before proving Proposition 4.13. This is the reason why it works in Ponsinet's setting of supersingular abelian varieties and in our case too. 
\end{remark}
\begin{remark}\label{rem:gen} In this remark, we recall some general fact.
    If we have a $\mathcal{O}$-lattice  $T$, and we have a coleman map $$\mathrm{col}: H^1_{Iw}(\Q_p(\mu_{p^\infty}), T) \rightarrow \Lambda_{\mathcal{O}}(\Gamma_0),$$ then by taking the orthogonal complement of the kernel of $\mathrm{col}$ under the  pairing $$H^1(\Q_p(\mu_{p^\infty}), M^*) \times H^1_{Iw}(\Q_p(\mu_{p^\infty}), T) \rightarrow \frac{\mathcal{O}[\frac{1}{p}]}{\mathcal{O}} $$
then one gets a submodule  of $H^1(\Q_p(\mu_{p^\infty}), M^*)$ which we denote by $H^1_I(\Q_p(\mu_{p^\infty}), M^*)$ (here $M^*=T^* \otimes \Q_p/\Z_p$ where  $T^*= Hom_{\mathcal{O}}(T, \mathcal{O}(1))$ (see \cite[Section 1.6]{ponsinet})). It follows that the Pontryagin dual of this local condition $H^1_I(\Q_p(\mu_{p^\infty}), M^*)$ is isomorphic to the image of $\mathrm{col}$. For a reference see \cite[p. 1276, line 4, or Remark 4.2]{HatleyLei19}.
\end{remark}
\noindent{\textit{Proof of Theorem  \ref{prop:finite}:}} The following argument is an adaptation of \cite[Proposition 4.14]{Greenberg}. The argument, being algebraic in nature, works in the supersingular case as well (see \cite[Proof of Theorem 2.1 after Remark 2.9]{ponsinet} or \cite[Proof of Theorem 3.1, p. 1276]{HatleyLei19}). We repeat the argument for completeness.

Note that it will be sufficient to show Theorem \ref{prop:finite} for $\Sel_{\fS}(\bA_{f,\psi,s}/\Q_\infty)$ for some $s$.   Lemma \ref{lem: first} and Lemma \ref{lemma:second} gives that for all but finitely many $s \in \Z$ the map 
 $$H^1(\Q_\infty,\bA_{f, \psi,s})^{\Gamma} \rightarrow  \bigoplus_{v}\mathcal{P}_{v}(\bA_{f, \psi,s}/\Q_\infty )^{\Gamma} $$
     is surjective. Again by Lemma \ref{lem: first} we have the exact sequence 
     $$0 \rightarrow \Sel_{\fS}(\bA_{f,\psi,s}/\Q_\infty) \rightarrow H^1(\Q_\infty,\bA_{f, \psi,s}) \rightarrow  \bigoplus_{v}\mathcal{P}_{v}(\bA_{f, \psi,s}/\Q_\infty ) \rightarrow 0. $$
     On taking $\Gamma$-coinvariants we obtain the following long exact sequence.
     \begin{equation}\label{eq:exact_seq}
     H^1(\Q_\infty,\bA_{f, \psi,s})^\Gamma\rightarrow \bigoplus_{v}\mathcal{P}_{v}(\bA_{f, \psi,s}/\Q_\infty )^{\Gamma} \rightarrow \Sel_{\fS}(\bA_{f,\psi,s}/\Q_\infty)_{\Gamma} \rightarrow H^1(\Q_\infty,\bA_{f, \psi,s})_{\Gamma}.
     \end{equation}
    
    Using Poitou-Tate exact sequence \cite[Proposition A.3.2]{PR95}, $H^2(\Q_\infty, \bA_{f,\psi,s})$ injects into $\varinjlim \oplus H^2((\Q_{(n)})_v, \bA_{f,\psi,s}).$ By local Tate-duality, this is isomorphic to $\varprojlim \oplus H^0((\Q_{(n)})_v, \bM_{-s}).$ But this is zero (arguments as in \cite[Corollary 2.8]{HatleyLei19}). Using Hochschild-Serre spectral sequence, as in \cite[Proposition 3.2]{HatleyLei19}, one can deduce that $H^1(\Q_\infty, \bA_{f, \psi, s})$ has no nontrivial submodule of finite index and this is true for all $s\in \Z$. This implies that, for suitably chosen $s$ $$H^1(\Q_\infty,\bA_{f, \psi,s})_{\Gamma}=0$$ (cf. \cite[line -13, page 103]{Greenberg}). The proof follows by noting that the first map in \eqref{eq:exact_seq} is surjective, hence the last map in \eqref{eq:exact_seq} is injective, but $H^1(\Q_\infty,\bA_{f, \psi,s})_{\Gamma}=0$. Therefore, $\Sel_{\fS}(\bA_{f,\psi,s}/\Q_\infty)_{\Gamma}$ must be trivial. \qed
\vspace{.2cm}

From the following exact sequence 
$$0 \rightarrow \bA_{f, \psi}[\mathfrak{P}] \rightarrow \bA_{f, \psi} \rightarrow \bA_{f, \psi} \rightarrow 0 $$
we obtain that the sequence 
$$ 0 \rightarrow H^0(\Q_\infty, \bA_{f, \psi})/\mathfrak{P} \rightarrow H^1(\Q_\infty, \bA_{f, \psi}[\mathfrak{P}] ) \rightarrow H^1(\Q_\infty,\bA_{f, \psi})[\mathfrak{P}] \rightarrow 0. $$
As $H^0(\Q_\infty, \bA_{f, \psi})=0$ by \textbf{(invS)}, we obtain $$H^1(\Q_\infty, \bA_{f, \psi}[\mathfrak{P}] ) \cong H^1(\Q_\infty,\bA_{f, \psi})[\mathfrak{P}]. $$ The same proof also gives that, for a prime $v \mid p$, 
$$H^1((\Q_\infty)_v, \bA_{f, \psi}[\mathfrak{P}] ) \cong H^1((\Q_\infty)_v,\bA_{f, \psi})[\mathfrak{P}]. $$
Define $$H^1_{\fS}((\Q_\infty)_v, \bA_{f, \psi}[\mathfrak{P}] ):=H^1_{\fS}((\Q_\infty)_v, \bA_{f, \psi})[\mathfrak{P}].$$
It is also well-known that $H^1_{\ur}((\Q_\infty)_v,\bA_{f, \psi}[\mathfrak{P}])=H^1_{\ur}((\Q_\infty)_v,\bA_{f, \psi})[\mathfrak{P}]$ for primes $v\nmid p$ and $v\not\in \Sigma.$ 
\begin{definition}\label{def:imprimitive} The $\Sigma_0$-imprimitive signed Selmer group $\Sel_{\fS}^{\Sigma_0}(\bA_{f,\psi}/\Q_\infty)$ for the Galois module $\bA_{f,\psi}$ is define as the kernel of the following map 
$$H^1(\Q_\infty,\bA_{f, \psi}) \rightarrow \prod_{v \mid p} \frac{H^1((\Q_\infty)_v,\bA_{f, \psi})}{H_{\fS}^1((\Q_\infty)_v,\bA_{f, \psi})}  \times \prod_{v \in \Sigma\backslash \Sigma_0}\frac{H^1((\Q_\infty)_v,\bA_{f, \psi})}{H^1_{\ur}((\Q_\infty)_v,\bA_{f, \psi})}.$$
The $\Sigma_0$-imprimitive signed Selmer group $\Sel_{\fS}^{\Sigma_0}(\bA_{f,\psi}[\mathfrak{P}]/\Q_\infty)$ for the Galois module $\bA_{f,\psi}[\mathfrak{P}]$ is define as the kernel of the following map 
$$H^1(\Q_\infty,\bA_{f, \psi}[\mathfrak{P}]) \rightarrow \prod_{v \mid p} \frac{H^1((\Q_\infty)_v,\bA_{f, \psi}[\mathfrak{P}])}{H_{\fS}^1((\Q_\infty)_v,\bA_{f, \psi}[\mathfrak{P}])}  \times \prod_{v \in \Sigma\backslash \Sigma_0}\frac{H^1((\Q_\infty)_v,\bA_{f, \psi}[\mathfrak{P}])}{H^1_{\ur}((\Q_\infty)_v,\bA_{f, \psi}[\mathfrak{P}])}.$$  
\end{definition}
By the discussion before definition \ref{def:imprimitive}, it follows that the local conditions defining the Selmer groups $\Sel_{\fS}^{\Sigma_0}(\bA_{f,\psi}/\Q_\infty)[\mathfrak{P}]$ and $ \Sel_{\fS}^{\Sigma_0}(\bA_{f,\psi}[\mathfrak{P}]/\Q_\infty)$ are the same. Hence there is a $\Lambda$-module isomorphism 
\begin{equation}\label{eq:jishnu18}
    \Sel_{\fS}^{\Sigma_0}(\bA_{f,\psi}/\Q_\infty)[\mathfrak{P}] \cong  \Sel_{\fS}^{\Sigma_0}(\bA_{f,\psi}[\mathfrak{P}]/\Q_\infty).
\end{equation}
\begin{lemma}\label{lem:4.9}
    Assume that $\Sel_{\fS}(\bA_{f, \psi})$ and $\Sel_{\fS}(\bJ_{f, \psi^{-1},-t})$ are cotorsion $\Lambda$-modules for some fixed $t\in \Z$. Then the Selmer group $\Sel^{\Sigma_0}_{\fS}(\bA_{f, \psi})$ is also $\Lambda$-cotorsion and 
     $$\mu(\Sel_{\fS}(\bA_{f, \psi})) =\mu(\Sel^{\Sigma_0}_{\fS}(\bA_{f, \psi})).$$
\end{lemma}
\begin{proof}
    The Pontryagin dual of $\Sel^{\Sigma_0}_{\fS}(\bA_{f, \psi})$ is a quotient of $\Sel_{\fS}(\bA_{f, \psi})$ and hence $\Sel^{\Sigma_0}_{\fS}(\bA_{f, \psi})$ is cotorsion. As the Selmer group $\Sel_{\fS}(\bJ_{f, \psi^{-1},-t})$ is cotorsion, by Lemma \ref{lem: first}, the global to local map defining the Selmer group $\Sel_{\fS}(\bA_{f, \psi})$ is surjective. Hence the global to local map defining the Selmer group $\Sel^{\Sigma_0}_{\fS}(\bA_{f, \psi})$ is also surjective. Also, $H^1_{\ur}((\Q_\infty)_v, \bA_{f, \psi}))=0$ for $v \nmid p$ (cf. \cite[A.2.4]{PR95}). Therefore,
    \eqref{eq:Sel1} holds where the Selmer groups are replaced by the corresponding signed versions. The proof now follows from Lemma \ref{eq:Lemma1}.
\end{proof}
\begin{remark}\label{rem:one}
      Assume that $\Sel_{\fS}(\bA_{f, \psi})$ and $\Sel_{\fS}(\bJ_{f, \psi^{-1},-t})$ are cotorsion $\Lambda$-modules for some fixed $t\in \Z$. Then the imprimitive signed Selmer group
$\Sel^{\Sigma_0}_{\fS}(\bA_{f, \psi})$ contains no proper $\Lambda$-submodule of finite index. The proof is the same as the proof of  Theorem \ref{prop:finite} and hence omitted.
\end{remark}
Now suppose $f_i$ (for $i=1,2$) are modular forms of level $N_i$ and character $\varepsilon_i$ as in section \ref{sec:congruent} with isomorphic residual representations $\tilde{\rho_1} \cong \tilde{\rho_2}$. Suppose that $a_p(f_i)=0$ for both $i =1,2$. Let $\Sigma$ be a set of places of $\Q$ that contains $p, \infty$, the primes dividing $N_1N_2$ and  the primes dividing the conductor $N_\psi$ and as before let $\Sigma_0=\Sigma \backslash \{p,\infty\}.$

\begin{proposition}\label{prop:congprop}
    We have the following isomorphism as $\Lambda$-modules.
    $$\Sel_{\fS}^{\Sigma_0}(\bA_{f_1, \psi}[\mathfrak{P}]/\Q_\infty)\cong \Sel_{\fS}^{\Sigma_0}(\bA_{f_2, \psi}[\mathfrak{P}]/\Q_\infty).$$
\end{proposition}
\begin{proof}
    Clearly, we have the isomorphisms 
    $$H^1(\Q_\infty,\bA_{f_1, \psi}[\mathfrak{P}]) \cong  H^1(\Q_\infty,\bA_{f_2, \psi}[\mathfrak{P}]) \text{ and }\mathcal{P}_v(\bA_{f_1, \psi}[\mathfrak{P}]/\Q_\infty) \cong \mathcal{P}_v(\bA_{f_2, \psi}[\mathfrak{P}]/\Q_\infty)$$ for primes $v \nmid p$.
    \vspace{.2cm}

    Since $\bT_{f_1,\psi}/\mathfrak{P}\bT_{f_1,\psi} \cong \bT_{f_2,\psi}/\mathfrak{P}\bT_{f_2,\psi}$, this gives, by duality (cf. \cite[page 23, line 3]{BLV}),  the congruence $$M_{1,f_1, \psi^{-1}} /\mathfrak{P} M_{1,f_1, \psi^{-1}} \cong M_{1,f_2, \psi^{-1}} /\mathfrak{P} M_{1,f_2, \psi^{-1}} \text{ and } M_{2,f_1, \psi^{-1}} /\mathfrak{P} M_{2,f_1, \psi^{-1}} \cong M_{2,f_2, \psi^{-1}} /\mathfrak{P} M_{2,f_2, \psi^{-1}}.$$
    By the theory of Wach modules, these congruence implies that the Iwasawa cohomologies $\HIw(\Q_p(\mu_{p^\infty}),M_{j,f_1, \psi^{-1}})$ and $\HIw(\Q_p(\mu_{p^\infty}),M_{j,f_2, \psi^{-1}})$  are congruent modulo $\mathfrak{P}$ for $j=1,2$ (the isomorphism between Iwasawa cohomology and Wach module for crystalline representations is mentioned in \cite[Section 2.2, Section 2.3]{LLZ0} and for congruence result related to Wach modules, see \cite[Theorem IV.1.1]{berger04}). We note that we need the hypothesis $p\geq 2k+1$ here to implement \cite[Theorem IV.1.1]{berger04} (for a discussion of Hodge-Tate weights in the symmetric square case, see  \cite[last line in page 25, page 26]{BLV}). Therefore, we have the following congruence of the images of signed Coleman maps defined in the proof of Lemma \ref{lemma:second}.
    \begin{align}
       \im(\overline{\col}_{f_1}^?)/\mathfrak{P}\im(\overline{\col}_{f_1}^?)&\cong \im(\overline{\col}_{f_2}^?)/\mathfrak{P}\im(\overline{\col}_{f_2}^?) \text{ for } ?\in \{\pm,\bullet\}. \label{eq:align1} 
    \end{align}
    Recall that $\mathcal{S}$ was the set of pairs $\{(+,-),(+,\bullet), (-,\bullet)\}.$ If $\fS = (\clubsuit, \spadesuit) \in \mathcal{S}$,
    then the following Coleman map was defined in the proof of Lemma \ref{lemma:second}.
    \begin{align*}
       \col_{f_i}^{\fS}:\HIw(\Q_p(\mu_{p^\infty}), \bM_{f_i, \psi^{-1}}) &\rightarrow \Lambda_{\O}(\Gamma_0)^{\oplus 2} \\
       z &\mapsto \col_{f_i}^{\clubsuit} (z)\oplus \col_{f_i}^{\spadesuit} (z).
    \end{align*}

The Pontryagin dual of $H^1_{\fS}((\Q(\mu_{p^\infty}))_v,\bA_{f_i, \psi}[\mathfrak{P}])$ is isomorphic to $\im \col_{f_i}^{\fS}/\mathfrak{P}\im \col_{f_i}^{\fS}$. Therefore, using \eqref{eq:align1}, we deduce that 
$$H^1_{\fS}((\Q(\mu_{p^\infty}))_v,\bA_{f_1, \psi}[\mathfrak{P}]) \cong H^1_{\fS}((\Q(\mu_{p^\infty}))_v,\bA_{f_2, \psi}[\mathfrak{P}]). $$ 
Taking $\Delta$-invariance, one obtains $$H^1_{\fS}((\Q_\infty)_v,\bA_{f_1, \psi}[\mathfrak{P}])\cong H^1_{\fS}((\Q_\infty)_v,\bA_{f_2, \psi}[\mathfrak{P}])$$ and this completes the proof of the proposition.   
\end{proof}
\begin{theorem}\label{prop: lambda}
Assume that the hypothesis of Lemma \ref{lem:4.9} is true for both $f_1$ and $f_2$.  Then the $\mu$-invariant of $\Sel_{\fS}(\bA_{f_1, \psi}/\Q_\infty)$ vanishes if and only if the $\mu$-invariant of $\Sel_{\fS}(\bA_{f_2, \psi}/\Q_\infty)$ vanishes. When these $\mu$-invariants are trivial, then the imprimitive signed $\lambda$-invariants coincide, i.e. 
$$\lambda (\Sel^{\Sigma_0}_{\fS}(\bA_{f_1, \psi}/\Q_\infty))=\lambda (\Sel^{\Sigma_0}_{\fS}(\bA_{f_2, \psi}/\Q_\infty)).$$
\end{theorem}
\begin{proof}
    By Lemma \ref{lem:4.9}, the $\mu$-invariant of $\Sel_{\fS}(\bA_{f_i, \psi}/\Q_\infty)$ vanishes if and only if $\Sel_{\fS}^{\Sigma_0}(\bA_{f_i, \psi}/\Q_\infty)$ is cofinitely generated as an $\O$-module which is true if and only if $\Sel_{\fS}^{\Sigma_0}(\bA_{f_i, \psi}/\Q_\infty)[\mathfrak{P}]$ is finite. By \eqref{eq:jishnu18}, $\Sel_{\fS}^{\Sigma_0}(\bA_{f_i, \psi}/\Q_\infty)[\mathfrak{P}]$ is finite if and only if  $\Sel_{\fS}^{\Sigma_0}(\bA_{f_i, \psi}[\mathfrak{P}]/\Q_\infty)$ is finite.  Then Proposition \ref{prop:congprop} gives that $\Sel_{\fS}^{\Sigma_0}(\bA_{f_1, \psi}[\mathfrak{P}]/\Q_\infty)$ is finite if and only if $\Sel_{\fS}^{\Sigma_0}(\bA_{f_2, \psi}[\mathfrak{P}]/\Q_\infty)$ is finite. 
\vspace{.2cm}
Since the imprimitive Selmer group $\Sel^{\Sigma_0}_{\fS}(\bA_{f_i, \psi}/\Q_\infty)$ has no proper $\Lambda$-submodule of finite index, one case easily show that $$\lambda(\Sel^{\Sigma_0}_{\fS}(\bA_{f_i, \psi}/\Q_\infty))=\dim_{\mathbb{F}_p}\Sel_{\fS}^{\Sigma_0}(\bA_{f_i, \psi}/\Q_\infty)[\mathfrak{P}].$$ But $\Sel_{\fS}^{\Sigma_0}(\bA_{f_1, \psi}/\Q_\infty)[\mathfrak{P}]\cong \Sel_{\fS}^{\Sigma_0}(\bA_{f_2, \psi}/\Q_\infty)[\mathfrak{P}]$ and hence their $\mathbb{F}_p$-dimensions coincide. 
\end{proof}
Theorem  \ref{prop: lambda} yields the following theorem. The proof is the same as in the ordinary case, since it only involves computing the $\lambda$-invariants of $\mathcal{P}_v(\bA_{f_i,\psi}/\Q_\infty)$ for primes $v$ that lie above places of $\Sigma_0$ and hence different from $p$. 
  \begin{theorem}\label{Mainthm2} Assume the vanishing of the $\mu$-invariants as in Theorem \ref{prop: lambda}. 
    We have the congruence
    $$\lambda(\Sel_{\fS}(\bA_{f_1, \psi}/\Q_\infty)) +|\mathcal{S}_{f_1,\psi}|\equiv \lambda(\Sel_{\fS}(\bA_{f_2, \psi}/\Q_\infty))+|\mathcal{S}_{f_2,\psi}| \pmod{2}$$
    where $\mathcal{S}_{f_i,\psi}$ are the set of primes in $\Sigma$ defined as in proposition \ref{prop: mod2}.
\end{theorem}

\section{Data availability}
This article does not have any associated data. 
\bibliographystyle{alpha}
\bibliography{main}

@article {Hatley17,
    AUTHOR = {Hatley, Jeffrey},
     TITLE = {Rank parity for congruent supersingular elliptic curves},
   JOURNAL = {Proc. Amer. Math. Soc.},
  FJOURNAL = {Proceedings of the American Mathematical Society},
    VOLUME = {145},
      YEAR = {2017},
    NUMBER = {9},
     PAGES = {3775--3786},
      ISSN = {0002-9939,1088-6826},
   MRCLASS = {14H52 (11F80 11R23)},
  MRNUMBER = {3665032},
MRREVIEWER = {Rupam\ Barman},
       DOI = {10.1090/proc/13545},
       URL = {https://doi.org/10.1090/proc/13545},
}

@article {DiamondTaylor,
    AUTHOR = {Diamond, Fred and Taylor, Richard},
     TITLE = {Lifting modular mod {$l$} representations},
   JOURNAL = {Duke Math. J.},
  FJOURNAL = {Duke Mathematical Journal},
    VOLUME = {74},
      YEAR = {1994},
    NUMBER = {2},
     PAGES = {253--269},
      ISSN = {0012-7094,1547-7398},
   MRCLASS = {11F33 (11F80 11G18 11S37)},
  MRNUMBER = {1272977},
MRREVIEWER = {Fernando\ Q.\ Gouv\^{e}a},
       DOI = {10.1215/S0012-7094-94-07413-9},
       URL = {https://doi.org/10.1215/S0012-7094-94-07413-9},
}

@article {Flach,
    AUTHOR = {Flach, Matthias},
     TITLE = {A generalisation of the {C}assels-{T}ate pairing},
   JOURNAL = {J. Reine Angew. Math.},
  FJOURNAL = {Journal f\"{u}r die Reine und Angewandte Mathematik. [Crelle's
              Journal]},
    VOLUME = {412},
      YEAR = {1990},
     PAGES = {113--127},
      ISSN = {0075-4102},
   MRCLASS = {11G09 (11R23 11S25)},
  MRNUMBER = {1079004},
MRREVIEWER = {Vincent Fleckinger},
       DOI = {10.1515/crll.1990.412.113},
       URL = {https://doi.org/10.1515/crll.1990.412.113},
}

@article {RaySujathaVatsal,
    AUTHOR = {Ray, Anwesh and Sujatha, R. and Vatsal, Vinayak},
     TITLE = {Iwasawa invariants for symmetric square representations},
   JOURNAL = {Res. Math. Sci.},
  FJOURNAL = {Research in the Mathematical Sciences},
    VOLUME = {10},
      YEAR = {2023},
    NUMBER = {3},
     PAGES = {Paper No. 27, 43},
      ISSN = {2522-0144},
   MRCLASS = {11R23},
  MRNUMBER = {4603407},
       DOI = {10.1007/s40687-023-00388-w},
       URL = {https://doi.org/10.1007/s40687-023-00388-w},
}

@article {HatleyLei19,
    AUTHOR = {Hatley, Jeffrey and Lei, Antonio},
     TITLE = {Arithmetic properties of signed {S}elmer groups at
              non-ordinary primes},
   JOURNAL = {Ann. Inst. Fourier (Grenoble)},
  FJOURNAL = {Universit\'{e} de Grenoble. Annales de l'Institut Fourier},
    VOLUME = {69},
      YEAR = {2019},
    NUMBER = {3},
     PAGES = {1259--1294},
      ISSN = {0373-0956},
   MRCLASS = {11R23 (11F11 11F85 11R18)},
  MRNUMBER = {3986915},
MRREVIEWER = {Maria Valentino},
       URL = {http://aif.cedram.org/item?id=AIF_2019__69_3_1259_0},
}

@book {HidabookModular,
    AUTHOR = {Hida, Haruzo},
     TITLE = {Modular forms and {G}alois cohomology},
    SERIES = {Cambridge Studies in Advanced Mathematics},
    VOLUME = {69},
 PUBLISHER = {Cambridge University Press, Cambridge},
      YEAR = {2000},
     PAGES = {x+343},
      ISBN = {0-521-77036-X},
   MRCLASS = {11F80 (11F11 11F67 11S25)},
  MRNUMBER = {1779182},
MRREVIEWER = {Gebhard B\"{o}ckle},
       DOI = {10.1017/CBO9780511526046},
       URL = {https://doi.org/10.1017/CBO9780511526046},
}

@article {LZ16,
    AUTHOR = {Loeffler, David and Zerbes, Sarah Livia},
     TITLE = {Iwasawa theory for the symmetric square of a modular form},
   JOURNAL = {J. Reine Angew. Math.},
  FJOURNAL = {Journal f\"{u}r die Reine und Angewandte Mathematik. [Crelle's
              Journal]},
    VOLUME = {752},
      YEAR = {2019},
     PAGES = {179--210},
      ISSN = {0075-4102},
   MRCLASS = {11F67 (11F80 11G18 11R23)},
  MRNUMBER = {3975641},
MRREVIEWER = {Tobias Berger},
       DOI = {10.1515/crelle-2016-0064},
       URL = {https://doi.org/10.1515/crelle-2016-0064},
}

@article {GreenbergVatsal,
    AUTHOR = {Greenberg, Ralph and Vatsal, Vinayak},
     TITLE = {On the {I}wasawa invariants of elliptic curves},
   JOURNAL = {Invent. Math.},
  FJOURNAL = {Inventiones Mathematicae},
    VOLUME = {142},
      YEAR = {2000},
    NUMBER = {1},
     PAGES = {17--63},
      ISSN = {0020-9910},
   MRCLASS = {11R23 (11G05)},
  MRNUMBER = {1784796},
MRREVIEWER = {Jacques Tilouine},
       DOI = {10.1007/s002220000080},
       URL = {https://doi.org/10.1007/s002220000080},
}

@article {ponsinet,
    AUTHOR = {Ponsinet, Gautier},
     TITLE = {On the structure of signed {S}elmer groups},
   JOURNAL = {Math. Z.},
  FJOURNAL = {Mathematische Zeitschrift},
    VOLUME = {294},
      YEAR = {2020},
    NUMBER = {3-4},
     PAGES = {1635--1658},
      ISSN = {0025-5874},
   MRCLASS = {11G10 (11R18 11R23)},
  MRNUMBER = {4074053},
MRREVIEWER = {Junhwa Choi},
       DOI = {10.1007/s00209-019-02328-3},
       URL = {https://doi.org/10.1007/s00209-019-02328-3},
}

@article {BLV,
    AUTHOR = {B\"{u}y\"{u}kboduk, K\^{a}zim and Lei, Antonio and Venkat, Guhan},
     TITLE = {Iwasawa theory for symmetric squares of non-{$p$}-ordinary
              eigenforms},
   JOURNAL = {Doc. Math.},
  FJOURNAL = {Documenta Mathematica},
    VOLUME = {26},
      YEAR = {2021},
     PAGES = {1--63},
      ISSN = {1431-0635},
   MRCLASS = {11R23 (11F11 11R20)},
  MRNUMBER = {4234458},
MRREVIEWER = {Andreas Nickel},
}

@article {Shekhar16,
    AUTHOR = {Shekhar, Sudhanshu},
     TITLE = {Parity of ranks of elliptic curves with equivalent {${\rm
              mod}\,p$} {G}alois representations},
   JOURNAL = {Proc. Amer. Math. Soc.},
  FJOURNAL = {Proceedings of the American Mathematical Society},
    VOLUME = {144},
      YEAR = {2016},
    NUMBER = {8},
     PAGES = {3255--3266},
      ISSN = {0002-9939,1088-6826},
   MRCLASS = {14G25 (11F33 11R23 14H52)},
  MRNUMBER = {3503694},
MRREVIEWER = {Ernest\ Hunter\ Brooks},
       DOI = {10.1090/proc/12993},
       URL = {https://doi.org/10.1090/proc/12993},
}

@incollection {Greenberg,
    AUTHOR = {Greenberg, Ralph},
     TITLE = {Iwasawa theory for elliptic curves},
 BOOKTITLE = {Arithmetic theory of elliptic curves ({C}etraro, 1997)},
    SERIES = {Lecture Notes in Math.},
    VOLUME = {1716},
     PAGES = {51--144},
 PUBLISHER = {Springer, Berlin},
      YEAR = {1999},
  
}

@article {JhaMajumdarShekhar,
    AUTHOR = {Jha, Somnath and Majumdar, Dipramit and Shekhar, Sudhanshu},
     TITLE = {{$p^r$}-{S}elmer companion modular forms},
   JOURNAL = {Ann. Inst. Fourier (Grenoble)},
  FJOURNAL = {Universit\'e{} de Grenoble. Annales de l'Institut Fourier},
    VOLUME = {71},
      YEAR = {2021},
    NUMBER = {1},
     PAGES = {53--87},
      ISSN = {0373-0956,1777-5310},
   MRCLASS = {11F33 (11G40 11R23 11R34 11S25)},
  MRNUMBER = {4275863},
MRREVIEWER = {Jeffrey\ Hatley},
       DOI = {10.5802/aif.3392},
       URL = {https://doi.org/10.5802/aif.3392},
}

@article {EmertonPollackWeston,
    AUTHOR = {Emerton, Matthew and Pollack, Robert and Weston, Tom},
     TITLE = {Variation of {I}wasawa invariants in {H}ida families},
   JOURNAL = {Invent. Math.},
  FJOURNAL = {Inventiones Mathematicae},
    VOLUME = {163},
      YEAR = {2006},
    NUMBER = {3},
     PAGES = {523--580},
      ISSN = {0020-9910,1432-1297},
   MRCLASS = {11F33 (11F80 11F85 11R23)},
  MRNUMBER = {2207234},
MRREVIEWER = {Laurent\ N.\ Berger},
       DOI = {10.1007/s00222-005-0467-7},
       URL = {https://doi.org/10.1007/s00222-005-0467-7},
}

@article {LLZ0,
    AUTHOR = {Lei, Antonio and Loeffler, David and Zerbes, Sarah Livia},
     TITLE = {Wach modules and {I}wasawa theory for modular forms},
   JOURNAL = {Asian J. Math.},
  FJOURNAL = {Asian Journal of Mathematics},
    VOLUME = {14},
      YEAR = {2010},
    NUMBER = {4},
     PAGES = {475--528},
      
}

@article {wiles,
    AUTHOR = {Wiles, Andrew},
     TITLE = {Modular elliptic curves and {F}ermat's last theorem},
   JOURNAL = {Ann. of Math. (2)},
  FJOURNAL = {Annals of Mathematics. Second Series},
    VOLUME = {141},
      YEAR = {1995},
    NUMBER = {3},
     PAGES = {443--551},
     
}

@article {PR95,
    AUTHOR = {Perrin-Riou, Bernadette},
     TITLE = {Fonctions {$L$} {$p$}-adiques des repr\'esentations
              {$p$}-adiques},
   JOURNAL = {Ast\'erisque},
  FJOURNAL = {Ast\'erisque},
    NUMBER = {229},
      YEAR = {1995},
     PAGES = {198},
     
}

@article {berger04,
    AUTHOR = {Berger, Laurent},
     TITLE = {Limites de repr\'esentations cristallines},
   JOURNAL = {Compos. Math.},
  FJOURNAL = {Compositio Mathematica},
    VOLUME = {140},
      YEAR = {2004},
    NUMBER = {6},
     PAGES = {1473--1498},
     
}

@ARTICLE{kobayashi03,
  author = {Kobayashi, Shin-ichi},
  title = {Iwasawa theory for elliptic curves at supersingular primes},
  journal = {Invent. Math.},
  year = {2003},
  volume = {152},
  pages = {1--36},
  number = {1},
  fjournal = {Inventiones Mathematicae}
}

@article {leiloefflerzerbes10,
    AUTHOR = {Lei, Antonio and Loeffler, David and Zerbes, Sarah Livia},
     TITLE = {Wach modules and {I}wasawa theory for modular forms},
   JOURNAL = {Asian J. Math.},
  FJOURNAL = {Asian Journal of Mathematics},
    VOLUME = {14},
      YEAR = {2010},
    NUMBER = {4},
     PAGES = {475--528},
      
}

@article {leiloefflerzerbes11,
    AUTHOR = {Lei, Antonio and Loeffler, David and Zerbes, Sarah Livia},
     TITLE = {Coleman maps and the {$p$}-adic regulator},
   JOURNAL = {Algebra Number Theory},
  FJOURNAL = {Algebra \& Number Theory},
    VOLUME = {5},
      YEAR = {2011},
    NUMBER = {8},
     PAGES = {1095--1131},
      
}
\end{document}